\documentclass{amsart}
\usepackage{amsmath,amssymb}

\usepackage{mathtools}
\usepackage{booktabs} 
\usepackage{adjustbox} 
\usepackage{color,soul}
\usepackage{bbm}
\usepackage{hyperref}
\hypersetup{colorlinks=true, urlcolor=blue, citecolor=blue, linkcolor=blue}
\usepackage{bm}
\usepackage{amsopn,amsthm}
\usepackage{doi}

\newtheorem{theorem}{Theorem}[section]
\newtheorem{lemma}[theorem]{Lemma}
\newtheorem{corollary}[theorem]{Corollary}

\theoremstyle{definition}

\theoremstyle{remark}
\newtheorem{remark}[theorem]{Remark}
\newcommand{\dif}{\mathop{}\!\mathrm{d}}  
\numberwithin{equation}{section}
\def\p{\partial}
\def\N{\mathbb N}
\def\R{\mathbb R}
\DeclareMathOperator{\supp}{supp}
\DeclareMathOperator*{\argmin}{arg\,min}

\newcommand{\dx}{\, \dif x}
\newcommand{\ds}{\, \dif s}
\newcommand{\dt}{\, \dif t}
\newcommand{\dtau}{\, \dif \tau}

\makeatletter
\@namedef{subjclassname@2020}{\textup{2020} Mathematics Subject Classification}
\makeatother

\begin{document}

\title[{Friction Threshold Recovery in a Tresca Model}]{Convergence Analysis for the Recovery of the Friction Threshold in a Scalar Tresca Model}
\author{Erik Burman}
\address{Department of Mathematics, University College London, 807b Gower Street, London, WC1E
6BT, United Kingdom}
\curraddr{}
\email{e.burman@ucl.ac.uk}
\thanks{E.B. was supported by the EPSRC grants EP/T033126/1 and
  EP/V050400/1. For the purpose of open access, the author has applied a Creative Commons Attribution (CC BY) licence to any Author Accepted Manuscript version arising.}

\author{Marvin Kn\"oller}
\address{Department of Mathematics and Statistics, University of Helsinki, P.O 68, 00014, Helsinki,
Finland}
\curraddr{}
\email{marvin.knoller@helsinki.fi}
\thanks{M.K. was supported by the Research Council of Finland (Flagship of Advanced
Mathematics for Sensing, Imaging and Modelling grant 359182).}

\author{Lauri Oksanen}
\address{Department of Mathematics and Statistics, University of Helsinki, P.O 68, 00014, Helsinki,
Finland}
\curraddr{}
\email{lauri.oksanen@helsinki.fi}
\thanks{L.O. was supported by the European
Research Council of the European Union, grant 101086697 (LoCal), and the Research Council of
Finland, grants 347715, 353096 and 359182. Views and opinions expressed are those of the authors only and do not necessarily reflect
those of the European Union or the other funding organizations.}

\author{Andreas Rupp}
\address{Department of Mathematics, Saarland University, DE-66123 Saarbr\"ucken,
Germany}
\curraddr{}
\email{andreas.rupp@uni-saarland.de}
\thanks{A.R. was supported by the DFG grant 577175348.}

\date{}

\begin{abstract}
We consider a scalar valued elliptic partial differential equation on a sufficiently smooth domain $\Omega$, subject to a regularized Tresca friction-type boundary condition on a subset $\Gamma$ of $\partial \Omega$.
The friction threshold, a positive function appearing in this boundary condition, is assumed to be unknown and serves as the coefficient to be recovered in our inverse problem.
Assuming that (i) the friction threshold 
lies in a finite dimensional space 
with known basis functions,
(ii) the right hand sides of the partial differential equation are known, and (iii) the solution to the partial differential equation on some small open subset $\omega \subset \Omega$ is available, we develop an iterative computational method for the recovery of the friction threshold. This algorithm is simple to implement and is based on piecewise linear finite elements.
We show that the proposed algorithm converges in second order to a function $a_h$ and, moreover, that $a_h$ converges in second order in the finite element's mesh size $h$ to the true (unknown) friction threshold.
We highlight our theoretical results by simulations that confirm our rates numerically.
\end{abstract}
\maketitle
\section{Introduction}
Understanding friction is key to the modelling of many mechanical systems in engineering and the geosciences. For example, the prediction of earthquakes or, in soil–structure interaction, landslides relies on modelling the friction forces on fault lines; modelling the flow of glaciers requires an accurate description of the friction between the ice and the bedrock. Yet the mathematical models of friction remain poorly understood. For example, the well‑known Coulomb friction model has not been shown to be well‑posed, whereas this is the case for the physically less realistic Tresca friction law \cite{DL76}.
In view of the importance of friction in applications, the numerical approximation of friction is an active research area; see \cite{CHR23} and references therein. 

Of crucial importance in all models is the friction threshold. This is the critical limit, defined by the friction coefficient, that determines whether the interface remains in stick or transitions to slip under given mechanical loading conditions. Hence, determination of the friction threshold is a necessary step for the practical applicability of friction models.
In most situations, the friction coefficient cannot be directly measured, since the contact surfaces are not accessible. It is therefore of interest to determine the friction threshold by inverse identification using measurements on an accessible part of the body, or internal non‑intrusive measurements obtained by scanning technology. To determine the friction coefficient from such measurements, one must solve an inverse problem subject to a partial differential equation modelling the mechanics of the system. This is a highly nontrivial task.
In most cases, a linearised friction model in the form of a Robin condition is used within a data‑fitting optimisation framework (see, for instance, \cite{PZSHG12}) . Here the friction threshold is represented simply by the Robin coefficient. While there is a relatively rich mathematical literature on the determination of Robin coefficients \cite{BurCenJinZhou25, CEJ04, CFM12, Harr21, JCELMP08,  JZ09, JZ10}, the inverse problem associated with a nonlinear friction law, or more generally with nonlinear Robin‑type conditions, appears to have received less attention \cite{AS07, SL09, VX98}. In particular there seem to be no works that have considered the numerical analysis of inverse problems subject to nonlinear Robin conditions.

In this work, we consider the reconstruction of the friction threshold -- i.e., a positive coefficient weighting a nonlinear interface law -- using measurements in a subset of the bulk domain. For the modelling, we employ a regularised Tresca friction law in a scalar model problem, leading to a nonlinear Robin boundary condition. Although this is a simplification compared to the classical variational inequality formulation, it is not necessarily less physically relevant \cite{KikOd88}.
Following the ideas of \cite{BurKnoOks25}, we regularise the inverse problem by assuming that the coefficient lies sufficiently close to an a priori known finite‑dimensional space. This assumption allows us to prove a Lipschitz stability estimate for functions in $H^1$ that are solutions to a certain linearised problem and that satisfy a finite‑dimensionality condition on the boundary (Theorem \ref{thm:stability}). Applying the theory of nonlinear approximation due to Keller \cite{Keller75}, this stability estimate, combined with a careful finite element analysis of the nonlinear Robin problem, allows us to prove (Theorem \ref{thm:main}):
(i) local uniqueness of the reconstruction;
(ii) second‑order convergence of the friction threshold for piecewise affine finite element discretisations;
(iii) convergence of Newton's method for initial guesses sufficiently close to the solution.
The theoretical results are validated on a set of numerical examples.
\subsection{The simplified and regularized Tresca model}
Let $\Omega_1$ and $\Omega_2$ be two bounded, connected and $C^3$-smooth domains in $\R^d$, $d\in \{2,3\}$ with $\Omega_1 \subset \subset \Omega_2$ and let $\Omega = \Omega_2 \setminus \overline{\Omega_1}$. We denote the boundary of $\Omega_1$ by $\Gamma$, i.e., $\partial \Omega_1 = \Gamma$ and the boundary of $\Omega_2$ by $\Gamma_0$.
For some sufficiently regular functions $f$ and $a$, where $a$ is positive, we consider the partial differential equation that asks to determine the function $u$ satisfying 
\begin{subequations}\label{eq:tresca}
\begin{align}
\Delta u \, &= \, f  \quad \text{in } \Omega\, , \\ \qquad 
|\partial_\nu u| \, \leq \, a\,   \text{ and }\,  u \partial_\nu u + a|u| \, &= \, 0 \quad \text{on } \Gamma\,  , \label{eq:trescaIntro} \\ 
u \, &= \, 0 \quad \text{on } \Gamma_0\, .
\end{align}
\end{subequations}
This is a scalar-valued, idealized model appearing in the study of friction problems.
The boundary condition on $\Gamma$ is the scalar-valued Tresca friction law (see also \cite[Sec.\@ 1.3]{GloLionsTre81}).
It is perhaps the simplest model for a stick-slip boundary condition: The region in which $u=0$ is the stick region, while the region in which $ u \neq 0$ is the slip region.
Observe that on $\Gamma$ the property $|\partial_\nu u| < a$ implies $u = 0$ since $u \partial_\nu u + a|u|=0$. 
Therefore a necessary condition for $u\neq 0$ is that $|\partial_\nu u| = a$.
The function $a$ takes the role of the friction threshold. 
It is the quantity that we assume to be unknown later on and that we aim to reconstruct.

Denoting by $\partial_c$ the subgradient of a convex function, one can write the Tresca boundary condition \eqref{eq:trescaIntro} on $\Gamma$ as
\begin{align}\label{eq:sub}
\partial_\nu u + a \partial_c(|u|) \ni 0 \, .
\end{align}
Regularity theory for problem \eqref{eq:tresca} is typically obtained via a mollification argument applied to \eqref{eq:sub}.
Writing $j(v) = |v|$ for the absolute value, the function $j$ is replaced by a suitable smooth
approximation $j_\varepsilon$, see e.g. \cite[Ex.\@ 5, p.\@ 38]{Bre72}. Further, writing $\beta_\varepsilon$ for the derivative of $j_\varepsilon$, the mollified problem reduces to
\begin{subequations}\label{eq:nonlin_robin_tresca}
\begin{align}
\Delta u \, &= \, f \quad \text{in } \Omega\,  ,\\
\partial_\nu u + a \beta_\varepsilon(u) \, &= \, 0 \quad \text{on } \Gamma\, , \quad  \\
u \, &= \, 0 \quad \text{on } \Gamma_0\, .
\end{align}
\end{subequations}
As $\varepsilon \to 0$ the function $j_\varepsilon$ resembles $j$. 
From an engineering perspective it is reasonable to work with the mollified model, see the discussion in \cite[Sec.\@ 10.4]{KikOd88} in an elastostatic context.
Accordingly, the nonlinear partial differential equation \eqref{eq:nonlin_robin_tresca} will be the center of our studies.
\section{Well-posedness and stability for the nonlinear model problem}
Let $\Omega$, $\Gamma$ and $\Gamma_0$ be defined as in the previous section.
Moreover, let $\beta \in C^3(\R)$ be some non-decreasing function, for which $\beta'\in C_b^2(\R)$ and $\beta(0)=0$.
The requirement for $\beta'$ to be in $C_b^2(\R)$ guarantees in particular that $\beta$ and its derivatives have a global Lipschitz constant.
Furthermore, let $a \in L^\infty(\Gamma)$ satisfy $0<  a_0 \leq a$.

Let $H_\diamond^1(\Omega) = \{ u \in H^1(\Omega) \, : \, u|_{\Gamma_0} = 0 \}$.
For $f \in H^1(\Omega)'$ and $g \in H^{-1/2}(\Gamma)$
we consider the nonlinear Robin-type problem, to determine $u\in H_\diamond^1(\Omega)$ such that
\begin{subequations}\label{eq:nonlin_robin}
\begin{align}
\Delta u \, &= \, f \quad \text{in } \Omega\,  ,\\
\partial_\nu u + a \beta(u) \, &= \, g \quad \text{on } \Gamma\,  \label{eq:nonlin_robin_weakbc}
\end{align}
\end{subequations}
is satisfied in the variational sense. The requirement $u=0$ on $\Gamma_0$ is included in the space $H_\diamond^1(\Omega)$ and will not be highlighted additionally.
The weak formulation to \eqref{eq:nonlin_robin} is to determine $u \in H_\diamond^1(\Omega)$ such that
\begin{align}\label{eq:nonlin_robin_weak}
b(u,v) \, = \, \ell_{f,g}(v) \quad \text{for all } v \in H_\diamond^1(\Omega)\, ,
\end{align}
where $b: H_\diamond^1(\Omega)^2 \to \R$ and $\ell_{f,g}: H_\diamond^1(\Omega) \to \R$ are defined by
\begin{align}\label{eq:bandl}
b(u,v)  =  \int_{\Omega} \nabla u \cdot \nabla v \dx + \int_{\Gamma} a \beta(u) v \ds\, , \; \; \ell_{f,g}(v)  =  -\int_\Omega fv \dx + \int_{\Gamma}gv \ds \, .
\end{align}
Note that $b$ is a functional that is nonlinear in the first and linear in the second variable.
Any function in $H_\diamond^1(\Omega)$ vanishes on $\Gamma_0$. Therefore, a compactness argument yields that
\begin{align}\label{eq:poincare-type-inequ}
\Vert u\Vert_{H^1(\Omega)} \, \lesssim \, \Vert \nabla u \Vert_{L^2(\Omega)} \quad \text{for all } u \in H_{\diamond}^1(\Omega)\, .
\end{align}
Existence and uniqueness of the solution $u \in H_\diamond^1(\Omega)$ satisfying \eqref{eq:nonlin_robin_weak} follows by the Browder--Minty theorem. We summarize this in the theorem below.
\begin{theorem}\label{thm:exuni}
There is a unique $u \in H_\diamond^1(\Omega)$ satisfying \eqref{eq:nonlin_robin_weak}. Moreover,
\begin{align}\label{eq:ubound}
\Vert u \Vert_{H^1(\Omega)} \, \lesssim \, \Vert f \Vert_{H^1(\Omega)'} + \Vert g \Vert_{H^{-1/2}(\Gamma)}\, .
\end{align}
\end{theorem}
\begin{proof}
The variational equality \eqref{eq:nonlin_robin_weak} is (by the Riesz representation theorem) equivalent to a nonlinear operator equation, which is to determine $u \in H_\diamond^1(\Omega)$ that solves
\begin{align*}
&Au \, = \, L\, , \quad \text{where }\\
  b(u,v) \, = \, \langle Au, v \rangle_{H^1(\Omega)} \quad &\text{and} \quad \ell_{f,g}(v) \, = \, \langle L, v \rangle_{H^1(\Omega)} \quad \text{for all } u,v \in H_\diamond^1(\Omega)\, .
\end{align*}
The continuity of $\beta$ yields that $A$ is continuous.
We show that the nonlinear operator $A$ is strongly monotone that is,
\begin{align}\label{eq:strong_mon}
\langle Au- Av, u-v \rangle_{H^1(\Omega)} \, \gtrsim \Vert u-v \Vert_{H^1(\Omega)}^2 \quad \text{for all } u,v \in H_\diamond^1(\Omega)\, .
\end{align}
Then, by the Browder--Minty theorem (see e.g. \cite[Thm.\@ 26.A]{Zeid902b}) the operator $A^{-1}$ exists and is Lipschitz continuous. 
Moreover, due to the assumption that $\beta(0) = 0$ one finds that $u=0$ is a fixed point of $A$ and so, \eqref{eq:ubound} follows since
\begin{align*}
\Vert u \Vert_{H^1(\Omega)} \, \lesssim \, \Vert A^{-1}L - A^{-1}0 \Vert_{H^1(\Omega)} \, \lesssim \, \Vert L \Vert_{H^1(\Omega)'} \, \lesssim \,  \Vert f \Vert_{H^1(\Omega)'} + \Vert g \Vert_{H^{-1/2}(\Gamma)}\, .
\end{align*}
To show \eqref{eq:strong_mon}, let $u,v \in H_\diamond^1(\Omega)$. Then, we see that
\begin{align}\label{eq:comp1}
\langle Au- Av, u-v \rangle_{H^1(\Omega)}
= \Vert \nabla(u-v) \Vert_{L^2(\Omega)}^2 + \int_{\Gamma} a (\beta(u)-\beta(v)) (u-v) \ds \, .
\end{align}
Due to our assumption that $\beta$ is non-decreasing, the second term in \eqref{eq:comp1} is always positive.
Moreover, due to \eqref{eq:poincare-type-inequ} we find that
\begin{align*}
\Vert u-v \Vert_{H^1(\Omega)} \, \lesssim \, \Vert \nabla(u-v) \Vert_{L^2(\Omega)}\, .
\end{align*}
This yields \eqref{eq:strong_mon} and concludes the proof.
\end{proof}
With Theorem \ref{thm:exuni} the existence and uniqueness of a solution to \eqref{eq:nonlin_robin_weak} is secured.
An additional smoothness assumption $f \in L^2(\Omega)$ and $g \in H^{1/2}(\Gamma)$ is enough to guarantee that $u \in H^2(\Omega) \cap H_\diamond^1(\Omega)$ with norm dependency on $f$ and $g$. 
This was proven in \cite[Thm.\@ I.9, p.\@ 40]{Bre72} (with $g=0$, but the more general case follows similarly). 
For our purposes this regularity for $u$ is not sufficient. For our upcoming finite element analysis we require that $\beta'(u) \in C^1(\Gamma)$.
Therefore, we assume from now on and throughout this whole work that $a \in C^2(\Gamma)$ and that $f \in H^1(\Omega)$ and $g \in H^{3/2}(\Gamma)$.
Then, we can show the following theorem. The proof can be found in the appendix.
\begin{theorem}\label{thm:howp}
Let $f \in H^1(\Omega)$, $g \in H^{3/2}(\Gamma)$ and $a \in C^2(\Gamma)$. Then, the solution $u$ to \eqref{eq:nonlin_robin_weak} satisfies $u \in H^{3}(\Omega) \cap H_\diamond^1(\Omega)$.
Moreover, we have the following bounds
\begin{subequations}
\begin{align}\label{eq:H2bound}
\Vert u \Vert_{H^{2}(\Omega)} \, &\lesssim \, \Vert f \Vert_{L^2(\Omega)} + \Vert g \Vert_{H^{1/2}(\Gamma)}\, , \\ 
\begin{split}
\Vert u \Vert_{H^{3}(\Omega)} \, &\lesssim \, \Vert f \Vert_{H^1(\Omega)} + \Vert g \Vert_{H^{3/2}(\Gamma)} \\
& \phantom{\lesssim \, }+ \Vert a \Vert_{C^2(\Gamma)}\Vert\beta \Vert_{C^2(u(\Gamma))} \Vert u \Vert_{H^2(\Omega)} \big(1 + \Vert u \Vert_{H^2(\Omega)}^{1/2} \big)\, .
\end{split}
\end{align}
\end{subequations}
\end{theorem}
To emphasize the dependence of solutions to \eqref{eq:nonlin_robin} on the parameter $a \in C^2(\Gamma)$ we write $u^{(a)}$ instead of $u$.
In our analysis, a linearized version of \eqref{eq:nonlin_robin} plays a fundamental role.
For given $\dot{f} \in H^1(\Omega)'$, $\dot{g} \in H^{-1/2}(\Gamma)$ and $u^{(a)} \in H_\diamond^1(\Omega) \cap H^{5/2+\varepsilon}(\Omega)$ solving \eqref{eq:nonlin_robin}  the linearized problem is to determine $\dot{w}^{(a)} \in H_\diamond^1(\Omega)$ solving
\begin{subequations}\label{eq:linRob}
\begin{align}
\Delta \dot{w}^{(a)} \, &= \, \dot{f} \quad \text{in } \Omega\,  ,\\
\partial_\nu \dot{w}^{(a)} + a \beta'(u^{(a)})\dot{w}^{(a)} \, &= \, \dot{g} \quad \text{on } \Gamma \label{eq:linRob2}\, .
\end{align}
\end{subequations}
Note that $u^{(a)} \in H^{5/2+\varepsilon}(\Omega)$ implies that $u^{(a)}|_{\Gamma} \in H^{2+\varepsilon}(\Gamma)$.
In particular, due to Sobolev embedding, we get that $u^{(a)}|_{\Gamma} \in C^1(\Gamma)$ for the spatial dimensions $d \in \{2,3\}$.
Consequently, $a\beta'(u^{(a)}) \in C^1(\Gamma)$.
Under the assumption that $\dot{f} \in L^2(\Omega)$ and $\dot{g} \in H^{1/2}(\Gamma)$ we find that $\dot{w}^{(a)} \in H^2(\Omega)$ since $\Omega$ is smooth. This can be seen by combining regularity theorems for Dirichlet and Robin problems (see, e.g.\@, \cite[Thm.\@ 2.4.2.5, 2.4.2.6]{Gris85}) with interior regularity for elliptic problems.
We introduce the bilinear form $\dot{b}[u]$ and the linear form $\dot{\ell}_{\dot{f}, \dot{g}}$ associated to \eqref{eq:linRob}, defined by
\begin{subequations}\label{def:dotbdotell}
\begin{align}
\dot{b}[u](\dot{w},v) \, &= \, \int_{\Omega} \nabla \dot{w} \cdot \nabla v \dx + \int_{\Gamma} a \beta'(u) \dot{w} v \ds \, , \\
\dot{\ell}_{\dot{f}, \dot{g}} (v) \, &= \, -\int_\Omega \dot{f} v\dx + \int_{\Gamma}  \dot{g} v \ds \, .
\end{align}
\end{subequations}

We turn now to our inverse problem, in which we want to recover an unknown friction threshold denoted by $\tilde{a} \in C^2(\Gamma)$.
Suppose that we know $f \in H^1(\Omega)$, $g \in H^{3/2}(\Gamma)$ as well as $q = u^{(\tilde{a})}|_\omega$ for some open $\omega \subset \Omega$.
A sketch of the geometry is shown in Figure~\ref{fig:sketch}.
\begin{figure}[t!]
\centering 
\includegraphics[scale=.35]{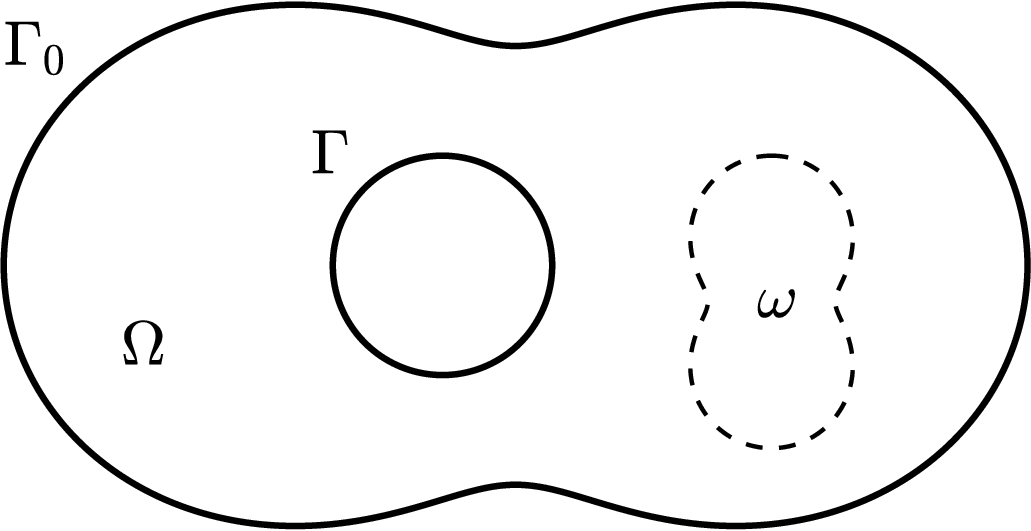}
\caption{Sketch of the geometrical configuration. For given $q=u^{(\tilde{a})}|_\omega \in L^2(\Omega)$, $f \in H^1(\Omega)$ and $g \in H^{3/2}(\Gamma)$ the aim is to reconstruct the function $\tilde{a} \in V_J \subset C^2(\Gamma)$.}
\label{fig:sketch}
\end{figure}
An approach based on optimal control theory would be to consider the problem to find
\begin{align*}
\argmin_{a \in C^2(\Gamma)} \Vert u^{(a)} - q \Vert_{L^2(\omega)}\, , \; \text{s.t. } u^{(a)} \text{ solves } b(u^{(a)}, v ) \, = \, \ell_{f,g}(v) \quad \text{for all } v \in H^1(\Omega)\, 
\end{align*}
with $b$ and $\ell_{f,g}$ defined as in \eqref{eq:bandl}.
Under the assumption that $\beta(u^{(\tilde{a})})$ does not vanish on any open subset of $\Gamma$, the unique continuation principle together with the boundary condition on $\Gamma$ from \eqref{eq:nonlin_robin_weakbc} guarantees that the recovery of $\tilde{a}$ is possible.
In numerical computation, however, we can neither minimize over all $C^2(\Gamma)$ nor obtain exact solutions to \eqref{eq:nonlin_robin_weak}.
Instead, $C^2(\Gamma)$ must be replaced by some $J$-dimensional subspace $V_J \subset C^2(\Gamma)$ and the exact solution by some finite dimensional approximation.
In this work we will consider finite element approximations.

Our approach here is similar to \cite{BurKnoOks25}. We cast the optimal control approach into an unconstrained optimization problem, which is to minimize the functional 
$\Theta : H_\diamond^1(\Omega) \times H_\diamond^1(\Omega) \times C^2(\Gamma) \to \R$ given by
\begin{align}\label{eq:theta}
\Theta(u,z,a) \, = \, \frac{1}{2}\Vert u - q \Vert_{L^2(\omega)}^2 + b(u, z ) - \ell_{f,g}(z) \, .
\end{align}
The search for saddle points for this functional leads to three conditions that arise by deriving $\Theta$ with respect to its three variables. These conditions are to find 
$u \in H_\diamond^1(\Omega)$, $z \in H_\diamond^1(\Omega)$ and $a \in C^2(\Gamma)$ such that
\begin{subequations}
\begin{align}
b(u,v) \, &= \, \ell_{f,g}(v) \quad \text{for all } v \in H_\diamond^1(\Omega)\, , \label{eq:c1}\\
\dot{b}[u](z,v) \, &= \, \dot{\ell}_{1_\omega (u - q), 0}(v) \quad \text{for all } v \in H_\diamond^1(\Omega)\, , \label{eq:c2} \\
\int_{\Gamma} \eta \beta(u) z \ds \, &= \, 0 \quad \text{for all } \eta \in C^2(\Gamma)\, . \label{eq:c3}
\end{align}
\end{subequations}
The first condition \eqref{eq:c1} simply means that the optimal $u$ should be a weak solution to \eqref{eq:nonlin_robin}. If this $u$ is found, it can be used in \eqref{eq:c2} to compute $z$, a weak solution to \eqref{eq:linRob} with $\dot{f} = 1_\omega(u-q)$ and $\dot{g} = 0$. Now that both $u$ and $z$ are known they can be used in \eqref{eq:c3}.
In summary, we want to find $a \in C^2(\Gamma)$ satisfying $F(a)=0$, where $\mathbf{F}(\cdot) \eta: C^2(\Gamma) \to \R$ 
\begin{align}\label{eq:idealFc}
\mathbf F(a)\eta \, = \, \int_{\Gamma} \eta \beta(u^{(a)}) z^{(a)} \ds\, \quad \text{for all } \eta \in C^2(\Gamma)\, ,
\end{align}
and $u^{(a)}$ and $z^{(a)}$ solve \eqref{eq:c1} and \eqref{eq:c2}, respectively.

By our previous argument, we cannot hope to satisfy $\mathbf{F}(a)\eta = 0$ for all $\eta \in C^2(\Gamma)$, numerically. 
We introduce a finite dimensional space $V_J \subset C^2(\Gamma)$ that is spanned by $J$ basis functions $\phi_1, \dots, \phi_J\in C^2(\Gamma)$ and assume that $\tilde{a} \in V_J$.
Moreover, we replace $u^{(a)}$ and $z^{(a)}$ by some finite element approximations $u_h^{(a)} \approx u^{(a)}$ and $z_h^{(a)} \approx z^{(a)}$ (more information is found in the next section) and instead of \eqref{eq:idealFc} we consider the problem to find roots of the functional $F_h: V_J \to \R^J$ defined by
\begin{align}\label{eq:Fh}
F_h(a) \, = \, [F_{h,1}(a), \dots, F_{h,J}(a)]\, , \quad \text{where} \quad F_{h,j}(a) \, = \, \int_{\Gamma} \phi_j \beta(u_h^{(a)}) z_h^{(a)} \ds\, . 
\end{align}

Before we continue with the definition and the convergence analysis of the finite element approximations we turn towards a stability theorem tailored to our nonlinear boundary condition.

Denote by $P:H^{-1/2}(\Gamma) \to \widetilde{V}_J$ the projection onto $\widetilde{V}_J\subset H^{1/2}(\Gamma)$, a finite dimensional subspace.
Moreover, let $Q = 1-P :  H^{-1/2}(\Gamma) \to H^{-1/2}(\Gamma)$.
\begin{theorem}\label{thm:stability}
Let $u \in H^{2+\varepsilon}(\Gamma)$. For any $v\in H_\diamond^1(\Omega)$ with
$\Delta v \in L^2(\Omega)$ it holds that
    \begin{align}\label{stability}
\Vert v \Vert_{H^1(\Omega)} \, \lesssim \, \Vert v \Vert_{L^2(\omega)} + \Vert Q(\p_\nu v + a\beta'(u)v)\Vert_{H^{-1/2}(\Gamma)} + \Vert\Delta v \Vert_{L^2(\Omega)} \, .
    \end{align}
\end{theorem}
The implicit constant in \eqref{stability} depends on the finite dimensional subspace $\widetilde{V}_J$. 
\begin{proof}
The proof uses ideas from \cite[Thm.\@ 2.2]{BurKnoOks25} and \cite[Proof of Thm.\@ 3.3]{BurOksZhi25}.
For $g \in \widetilde{V}_J$, let $w \in H_{\diamond}^1(\Omega)$ be the solution to 
\begin{subequations}\label{eq:linRobp}
\begin{align}
\Delta w \, &= \, 0 \quad \text{in } \Omega\,  ,\\
\partial_\nu w + a \beta'(u)w  \, &= \, g \quad \text{on } \Gamma\, .\label{eq:linRob2p}
\end{align}
\end{subequations}
The boundary condition \eqref{eq:linRob2p} implies that $Q(\partial_\nu w + a \beta'(u) w) = 0$ on $\Gamma$.
The operator $A : \widetilde{V}_J \to L^2(\omega)$, $Ag = w|_\omega$, where $w$ solves \eqref{eq:linRobp} is injective, which follows from the unique continuation principle for elliptic equations.
As the linear subspace $\widetilde{V}_J$ is finite dimensional, there is a norm such that $A : \widetilde{V}_J \to A(\widetilde{V}_J)$ is an isomorphism.
Since all norms are equivalent on finite dimensional spaces, it holds that
\begin{align}\label{eq:normequ}
\Vert g \Vert_{H^{1/2}(\Gamma)} \, \lesssim \, \Vert Ag \Vert_{L^2(\omega)} \quad \text{for all } g \in \widetilde{V}_J\, .
\end{align}
For any $w \in H_\diamond^1(\Omega)$ satisfying \eqref{eq:linRobp} it holds that
\begin{align*}
A(P(\partial_\nu w + a \beta'(u) w)) \, = \, w|_\omega \;\; \text{since}\; \;
P(\partial_\nu w + a \beta'(u) w) \, = \, \partial_\nu w + a \beta'(u) w \quad \text{on } \Gamma
\end{align*}
and due to \eqref{eq:normequ} one finds
\begin{align*}
\Vert P(\partial_\nu w + a \beta'(u) w) \Vert_{L^2(\Gamma)} \, \lesssim \, \Vert w \Vert_{L^2(\omega)}
\, \lesssim \, \Vert w \Vert_{L^2(\Omega)} \, .
\end{align*}
This inequality may be used as the replacement for \cite[Eq.\@ (2.18)]{BurKnoOks25} in the proof of \cite[Thm.\@ 2.2]{BurKnoOks25} and the rest of the proof can be done analogously.
\end{proof}

\section{Finite element analysis for the nonlinear and linearized problems}
We consider a quasi-uniform triangulation $\mathcal{T}_h$ of the $C^3$-smooth domain $\Omega$ that allows for curved patches. 
In particular, we treat the situation, in which the triangulation is assumed to be exact.
On these meshes we consider a $C^0$ finite element space $V_h \subset H^1(\Omega)$, consisting of piecewise linear functions (after mapping of the curved tetrahedra or triangles to a reference element).
Moreover, we define $V_{\diamond,h} = V_h \cap H^1_{\diamond}(\Omega)$, the space of piecewise linear functions (after transform to the reference element) with vanishing trace on $\Gamma_0$. 
In \cite[Cor.\@ 5.2]{Ber89} it is shown that there is an interpolation operator $\Pi_h : H^2(\Omega) \to V_{\diamond,h}$ that satisfies the estimate
\begin{align}\label{eq:interpest}
\Vert u - \Pi_hu \Vert_{L^2(\Omega)} + h \Vert u - \Pi_hu \Vert_{H^1(\Omega)} \, \lesssim \, h^2 \Vert u \Vert_{H^2(\Omega)} \, .
\end{align}
The finite element approximation to \eqref{eq:nonlin_robin_weak} is to find $u_h^{(a)} \in V_{\diamond,h}$ such that
\begin{align}\label{eq:uh}
b(u_h^{(a)}, v) \, = \, \ell_{f,g}(v) \quad \text{for all } v \in V_{\diamond,h}\, .
\end{align}
The same argumentation as in Theorem~\ref{thm:exuni} gives that such a unique $u_h^{(a)} \in V_{\diamond,h}$ exists and that it satisfies the bound
\begin{align*}
\Vert u_h^{(a)} \Vert_{H^1(\Omega)} \, \lesssim \, \Vert f \Vert_{L^2(\Omega)} + \Vert g \Vert_{H^{-1/2}(\Gamma)}\, .
\end{align*}
For the approximation error in $H^1(\Omega)$ we obtain the following lemma.
\begin{lemma}\label{lem:H1}
Let $u^{(a)} \in H_\diamond^1(\Omega)$ be the unique solution to \eqref{eq:nonlin_robin_weak} and let $u_h^{(a)} \in V_{\diamond,h}$ solve \eqref{eq:uh}. Then, the approximation error in the $H^1$-norm can be bounded via
\begin{align}\label{eq:H1bound}
\Vert u^{(a)}-u_h^{(a)} \Vert_{H^1(\Omega)} \, \lesssim \, h \Vert u^{(a)} \Vert_{H^2(\Omega)}\, .
\end{align}
\end{lemma}
\begin{proof}
The (nonlinear) Galerkin identity corresponding to the problem \eqref{eq:nonlin_robin} reads
\begin{multline}\label{eq:nlgal}
b(u^{(a)}, u^{(a)}-u_h^{(a)}) - b(u_h^{(a)}, u^{(a)}-u_h^{(a)}) \\
\, = \, b(u^{(a)}, u^{(a)}-v_h) - b(u_h^{(a)}, u^{(a)}-v_h)
\end{multline}
for all $v_h \in V_{\diamond,h}$. Due to \eqref{eq:poincare-type-inequ} the left hand side in \eqref{eq:nlgal} satisfies 
\begin{align*}
\Vert u^{(a)} - u_h^{(a)} \Vert_{H^1(\Omega)}^2 \, \leq \, b(u^{(a)}, u^{(a)}-u_h^{(a)}) - b(u_h^{(a)}, u^{(a)}-u_h^{(a)})\, .
\end{align*}
The right hand side of \eqref{eq:nlgal} can be bounded by using the Cauchy--Schwarz inequality together with the Lipschitz continuity of $\beta$ to get that
\begin{align*}
b(u^{(a)}, u^{(a)}-v_h) - b(u_h^{(a)}, u^{(a)}-v_h) \, \lesssim \, 
\Vert u^{(a)}- u_h^{(a)} \Vert_{H^1(\Omega)} \Vert u^{(a)}-v_h \Vert_{H^1(\Omega)} \, 
\end{align*}
with an implicit constant depending on $\Vert a \Vert_{C^1(\Gamma)}$ and $\Vert \beta' \Vert_{L^\infty(\R)}$. Consequently, we obtain an analogy to C\'ea's Lemma, which is
\begin{align*}
\Vert u^{(a)} - u_h^{(a)} \Vert_{H^1(\Omega)} \, \lesssim \, \inf_{v_h \in V_{\diamond,h}}\Vert u^{(a)}-v_h \Vert_{H^1(\Omega)} \, .
\end{align*}
Now, \eqref{eq:H1bound} can be concluded
by using the interpolation estimate in \eqref{eq:interpest}. 
\end{proof}
When studying the linearized problem \eqref{eq:linRob} we cannot assume to know $u^{(a)}$, the exact solution to the nonlinear problem \eqref{eq:nonlin_robin}.
Therefore, we replace $u^{(a)}$ in the left hand side of \eqref{eq:linRob2} by $u_h^{(a)}$ the finite element approximation defined in \eqref{eq:uh}.
The next lemma shows that the corresponding finite element approximation $\dot{w}_h^{(a)}$ still converges with respect to the same error bound in the mesh size to $\dot{w}^{(a)}$.
\begin{lemma}\label{lem:felinearized}
Let $\dot{w}^{(a)} \in H_\diamond^1(\Omega)$ and $\dot{w}_h^{(a)}\in V_{\diamond,h}$ be the unique functions satisfying
\begin{subequations}\label{eq:weakform}
\begin{align}
\dot{b}[u^{(a)}](\dot{w}^{(a)}, v) \, = \, \dot{\ell}_{\dot{f}, \dot{g}}(v)\quad \text{for all } v \in H_{\diamond}^1(\Omega)\, , \label{eq:udot} \\
\dot{b}[u_h^{(a)}](\dot{w}_h^{(a)}, v) \, = \, \dot{\ell}_{\dot{f}, \dot{g}}(v)\quad \text{for all } v \in V_{\diamond,h}\, \label{eq:uhdot}
\end{align}
\end{subequations}
for given $\dot{f} \in L^2(\Omega)$ and $\dot{g} \in H^{1/2}(\Gamma)$ where $\dot{b}$ and $\dot{\ell}$ are defined in \eqref{def:dotbdotell}. Then,
\begin{align}\label{eq:linearizedH1}
\Vert \dot{w}^{(a)} - \dot{w}_h^{(a)} \Vert_{H^1(\Omega)} \, \lesssim \, h \Vert \dot{w}^{(a)} \Vert_{H^2(\Omega)} \, .
\end{align}
The implicit constant in \eqref{eq:linearizedH1} depends on $a$ and on $u^{(a)}$.
\end{lemma}

\begin{proof}
Higher order regularity results yield that $\dot{w}^{(a)} \in H^2(\Omega) \cap H_\diamond^{1}(\Omega)$ and
\begin{align*}
\Vert \dot{w}^{(a)} \Vert_{H^2(\Omega)} \, \lesssim \,  \Vert \dot{f} \Vert_{L^2(\Omega)} + \Vert \dot{g} \Vert_{H^{1/2}(\Gamma)} \, .
\end{align*}
We abbreviate $\dot{e}^{(a)} = \dot{w}^{(a)}-\dot{w}_h^{(a)}$ and use \eqref{eq:weakform} to see that
\begin{multline}\label{eq:altgal}
\dot{b}[u^{(a)}](\dot{w}^{(a)},\dot{e}^{(a)}) - \dot{b}[u_h^{(a)}](\dot{w}_h^{(a)},\dot{e}^{(a)}) \\
\, = \, \dot{b}[u^{(a)}](\dot{w}^{(a)},\dot{w}^{(a)} - v) - \dot{b}[u_h^{(a)}](\dot{w}_h^{(a)},\dot{w}^{(a)} - v) \quad \text{for all } v \in V_{\diamond,h} \, .
\end{multline}
The left hand side of \eqref{eq:altgal} is
\begin{multline}\label{eq:lhs}
\Vert \nabla \dot{e}^{(a)} \Vert_{L^2(\Omega)}^2+ \int_{\Gamma}a \left( \beta'(u^{(a)})\dot{w}^{(a)} - \beta'(u_h^{(a)}) \dot{w}_h^{(a)} \right) \dot{e}^{(a)}\ds \\
\, = \, \Vert \nabla \dot{e}^{(a)} \Vert_{L^2(\Omega)}^2 + \int_{\Gamma} a \beta'(u_h^{(a)})(\dot{e}^{(a)})^2 \ds + \int_{\Gamma}a ( \beta'(u^{(a)}) - \beta'(u_h^{(a)}))\dot{w}^{(a)}\dot{e}^{(a)} \ds.
\end{multline}
The first term on the right hand side of \eqref{eq:lhs} can be bounded from below by $\Vert \dot{e}^{(a)}\Vert_{H^1(\Omega)}^2$ by using the Poincar\'e type inequality \eqref{eq:poincare-type-inequ}. 
The second term on the right hand side of \eqref{eq:lhs} is positive.
Returning to \eqref{eq:altgal} now shows that
\begin{align}\label{eq:spliterror}
\begin{split}
\Vert \dot{e}^{(a)} \Vert_{H^1(\Omega)}^2\, &\lesssim \, \left| \int_{\Gamma}a ( \beta'(u^{(a)}) - \beta'(u_h^{(a)}))\dot{w}^{(a)}\dot{e}^{(a)} \ds \right| \\
& \phantom{\lesssim \, } + \left|\int_\Omega \nabla \dot{e}^{(a)} \cdot \nabla (\dot{w}^{(a)} - v) \dx \right| \\
& \phantom{\lesssim \, } + \left|\int_{\Gamma} a \left( \beta'(u^{(a)})\dot{w}^{(a)} - \beta'(u_h^{(a)})\dot{w}_h^{(a)}\right) (\dot{w}^{(a)} - v) \ds \right| \\
\, &=: \, T_1 + T_2 + T_3\, .
\end{split}
\end{align}
We use H\"older's inequality with coefficients satisfying $1/2+1/4+1/4=1$, the embedding $H^{1/2}(\Gamma) \subset L^4(\Gamma)$ (see, e.g., \cite[Thm.\@ 6.7]{DiNPaVal12}), the fact that $\beta'$ has a global Lipschitz constant (since $\beta'\in C_b^2(\R)$) together with \eqref{eq:H1bound} and get that
\begin{align*}
T_1 \, &\lesssim \,  \Vert a \Vert_{L^\infty(\Gamma)} \Vert \beta'(u^{(a)}) - \beta'(u_h^{(a)}) \Vert_{L^2(\Gamma)} \Vert \dot{w}^{(a)} \Vert_{H^{1/2}(\Gamma)} \Vert \dot{e}^{(a)} \Vert_{H^{1/2}(\Gamma)}  \\ \, 
&\lesssim \, h \Vert a \Vert_{L^\infty(\Gamma)} \Vert u^{(a)} \Vert_{H^2(\Omega)} \Vert \dot{w}^{(a)} \Vert_{H^2(\Omega)} \Vert \dot{e}^{(a)} \Vert_{H^1(\Omega)} \, .
\end{align*}
For the second and third term we insert $v = \Pi_h \dot{w}^{(a)}$, where $\Pi_h$ is the interpolation operator satisfying \eqref{eq:interpest}. Using this estimate together with the Cauchy--Schwarz inequality shows that
\begin{align*}
T_2 \, \lesssim \, h \Vert \dot{e}^{(a)} \Vert_{H^1(\Omega)} \Vert\dot{w}^{(a)} \Vert_{H^2(\Omega)} \, .
\end{align*}
Finally, the third term can be bounded similarly as the terms before. Indeed
\begin{align*}
T_3 \, &\lesssim \, \left| \int_{\Gamma} a (\beta'(u^{(a)}) - \beta'(u_h^{(a)})) \dot{w}^{(a)} (\dot{w}^{(a)}-v)\ds \right| 
\\& \phantom{\lesssim } + \left| \int_{\Gamma} a (\beta'(u_h^{(a)}) - \beta'(u^{(a)})) \dot{e}^{(a)} (\dot{w}^{(a)}-v)\ds \right| + \left| \int_{\Gamma} a \beta'(u^{(a)}) \dot{e}^{(a)} (\dot{w}^{(a)}-v)\ds \right|  \\
& \lesssim \, h^2 \Vert a \Vert_{L^\infty(\Gamma)} \Vert {u}^{(a)} \Vert_{H^2(\Omega)} \Vert \dot{w}^{(a)} \Vert_{H^2(\Omega)}^2 \\
&\phantom{ \lesssim } + h\Vert a  \Vert_{L^\infty(\Gamma)} (h \Vert {u}^{(a)} \Vert_{H^2(\Omega)} + \Vert \beta'(u) \Vert_{L^\infty(\Gamma)}) \Vert \dot{e}^{(a)} \Vert_{H^1(\Omega)} \Vert \dot{w}^{(a)} \Vert_{H^2(\Omega)} \, .
\end{align*}
Collecting the estimates for $T_1, T_2$ and $T_3$ and returning to \eqref{eq:spliterror} yields
\begin{equation}\label{eq:depon}
\Vert \dot{e}^{(a)} \Vert_{H^1(\Omega)}^2 
\, \lesssim \, h \Vert \dot{w}^{(a)} \Vert_{H^2(\Omega)} 
\Vert \dot{e}^{(a)} \Vert_{H^1(\Omega)} + h^2 \Vert \dot{w}^{(a)} \Vert_{H^2(\Omega)}^2\, .
\end{equation}
The implicit constant in the inequality \eqref{eq:depon} depends on $\Vert a \Vert_{L^\infty(\Gamma)}$, $\Vert u^{(a)} \Vert_{H^2(\Omega)}$ and $\Vert \beta'(u^{(a)}) \Vert_{L^\infty(\Gamma)}$.
For any numbers $x,y,z\geq 0$, for which $x^2 \leq xy + z$ we find that $x^2 \leq y^2 + 2z$. Applying this to \eqref{eq:depon} finally yields \eqref{eq:linearizedH1}. 
\end{proof}

We apply Nitsche's trick to the linearized problem \eqref{eq:linRob} to obtain quadratic convergence in the mesh size $h$ when measuring the approximation error in the $L^2$-norm as shown in the next lemma.
\begin{lemma}\label{lem:L2}
Let $u^{(a)} \in H_\diamond^1(\Omega)$ be the unique solution to \eqref{eq:nonlin_robin_weak} and let $u_h^{(a)} \in V_{\diamond,h}$ solve \eqref{eq:uh}. Then, the approximation error in the $L^2$-norm can be bounded via
\begin{align}\label{eq:L2bound}
\Vert u^{(a)}-u_h^{(a)} \Vert_{L^2(\Omega)} \, \lesssim \, h^2 (\Vert u^{(a)} \Vert_{H^2(\Omega)} + \Vert u^{(a)}\Vert_{H^2(\Omega)}^2) \, .
\end{align}
\end{lemma}
\begin{proof}
We abbreviate $e^{(a)} = u^{(a)}-u_h^{(a)}$ and, as before, $\dot{e}^{(a)} = \dot{w}^{(a)}-\dot{w}_h^{(a)}$.
In \eqref{eq:weakform} let $\dot{f} = -e^{(a)}$ and $\dot{g}=0$.
Then, by \eqref{eq:linearizedH1}, for the finite element approximation $\dot{w}_h^{(a)}$ satisfying \eqref{eq:uhdot} we find
\begin{align}\label{eq:H1bounddot}
\Vert \dot{e}^{(a)} \Vert_{H^1(\Omega)} \, = \,  \Vert \dot{w}^{(a)} - \dot{w}_h^{(a)} \Vert_{H^1(\Omega)} \, \lesssim \, h \Vert \dot{w}^{(a)} \Vert_{H^2(\Omega)} \, \lesssim \, h \Vert e^{(a)} \Vert_{L^2(\Omega)}\, .
\end{align}
We use \eqref{eq:udot} and insert $e^{(a)}$ as a test function.
\begin{align}\label{eq:erre}
\begin{split}
\Vert e^{(a)} \Vert_{L^2(\Omega)}^2 \, &= \, \int_\Omega \nabla e^{(a)}\cdot \nabla \dot{w}^{(a)} \dx + \int_{\Gamma} a \beta'(u^{(a)}) \dot{w}^{(a)} e^{(a)} \ds \\
&= \, \int_\Omega \nabla e^{(a)}\cdot \nabla \dot{e}^{(a)} \dx + \int_\Omega \nabla e^{(a)}\cdot \nabla \dot{w}_h^{(a)} \dx \\
&\phantom{= \,} + \int_{\Gamma} a \beta'(u^{(a)}) \dot{e}^{(a)} e^{(a)} \ds +  \int_{\Gamma} a \beta'(u^{(a)}) \dot{w}_h^{(a)} e^{(a)} \ds \, .
\end{split}
\end{align}
We note that the Galerkin identity for the problem \eqref{eq:nonlin_robin_weak} reads
\begin{align}\label{eq:galeq}
\int_\Omega \nabla e^{(a)} \cdot \nabla v \dx + \int_{\Gamma}a( \beta(u^{(a)}) - \beta(u_h^{(a)})  ) v \ds \, = \, 0\quad \text{for all } v \in V_{\diamond,h} \, .
\end{align}
Moreover, we find that
\begin{align*}
\beta(u^{(a)}) - \beta(u_h^{(a)}) \, = \, \int_0^1 \beta'(u_h^{(a)} + te^{(a)}) e^{(a)} \dt \, .
\end{align*}
Inserting $v= \dot{w}_h^{(a)}$ in \eqref{eq:galeq} and subtracting the resulting equation from \eqref{eq:erre} yields
\begin{align}\label{eq:erre2}
\begin{split}
\Vert e^{(a)} \Vert_{L^2(\Omega)}^2 
\, &= \, \int_\Omega \nabla e^{(a)}\cdot \nabla \dot{e}^{(a)} \dx  + \int_{\Gamma} a \beta'(u^{(a)}) \dot{e}^{(a)} e^{(a)} \ds\\
&\phantom{= \,}  +  \int_{\Gamma} a \Big( \int_0^1  \beta'(u^{(a)}) - \beta'(u_h^{(a)} + te^{(a)})  \dt  \Big)\dot{w}_h^{(a)} e^{(a)} \ds \, .
\end{split}
\end{align}
We estimate \eqref{eq:erre2} term by term. For the first term we use the Cauchy--Schwarz inequality together with \eqref{eq:H1bound} and \eqref{eq:H1bounddot} and find that
\begin{align}\label{eq:T1}
\begin{split}
\left|\int_\Omega \nabla e^{(a)}\cdot \nabla \dot{e}^{(a)} \dx\right| \, &\lesssim \,  \Vert e^{(a)} \Vert_{H^1(\Omega)} \Vert \dot{e}^{(a)} \Vert_{H^1(\Omega)} \\
  &\lesssim \,  h^2 \Vert u^{(a)}\Vert_{H^2(\Omega)} \Vert e^{(a)} \Vert_{L^2(\Omega)}\, .
  \end{split}
\end{align}
In the same way we can bound the second term and see that
\begin{align}\label{eq:T2}
\left| \int_{\Gamma} a \beta'(u^{(a)}) \dot{e}^{(a)} e^{(a)} \ds \right| \, \lesssim \, h^2 \Vert u^{(a)}\Vert_{H^2(\Omega)} \Vert e^{(a)} \Vert_{L^2(\Omega)}\, .
\end{align}
In order to bound the third term we use that $\beta'$ is continuously differentiable and note that
\begin{align*}
\left| \int_0^1  \beta'(u^{(a)}) - \beta'(u_h^{(a)} + te^{(a)})  \dt \right| \, \lesssim \, \int_0^1 |u^{(a)} - (u_h^{(a)} + te^{(a)})| \dt \, = \, \frac{1}{2} |e^{(a)}|\, .
\end{align*}
We denote the third term on the right hand side of \eqref{eq:erre2} by $T_3$ and see by applying H\"older's inequality with coefficients satisfying $1/4+1/4+1/2=1$ and the embedding $H^{1/2}(\Gamma) \subset L^4(\Gamma)$ (see, e.g., \cite[Thm.\@ 6.7]{DiNPaVal12}) together with \eqref{eq:H1bound} and \eqref{eq:H1bounddot} that
\begin{align}
\begin{split}\label{eq:T3}
\left|  T_3 \right| \, &\lesssim \, \Vert e^{(a)} \Vert_{L^4(\Gamma)}^2 \Vert \dot{w}_h^{(a)} \Vert_{L^2(\Gamma)} \\
\, &\lesssim \Vert e^{(a)} \Vert_{H^1(\Omega)}^2\big( \Vert \dot{e}^{(a)} \Vert_{H^1(\Omega)} + \Vert \dot{w}^{(a)} \Vert_{H^1(\Omega)} \big)\, \, \\
\, &\lesssim \,  h^2 \Vert u^{(a)} \Vert_{H^2(\Omega)}^{2} \left( h \Vert e^{(a)} \Vert_{L^2(\Omega)} + \Vert e^{(a)} \Vert_{L^2(\Omega)} \right) \, .
\end{split}
\end{align}
Thus, using \eqref{eq:T1}, \eqref{eq:T2}  and \eqref{eq:T3} with \eqref{eq:erre2} yields
\begin{align*}
\Vert e^{(a)} \Vert_{L^2(\Omega)}^2 \, \lesssim \, h^2 (\Vert u^{(a)}\Vert_{H^2(\Omega)} + \Vert u^{(a)}\Vert_{H^2(\Omega)}^2) \Vert e^{(a)} \Vert_{L^2(\Omega)}\, .
\end{align*}
This shows \eqref{eq:L2bound}.
\end{proof}
Before continuing with theoretical results we consider a numerical example. 
For this, let $\Omega_1 = B_{1/2}(0) \subset \R^2$, $\Omega_2=B_1(0)\subset \R^2$ and $\Omega = \Omega_2\setminus \overline{\Omega_1}$.
Moreover, let $\beta = \beta_\varepsilon: \R \to \R$ be defined by
\begin{align}\label{eq:betadef}
\beta_\varepsilon(r) \, = \, 
\begin{cases}
-1 \quad &\text{for } r\leq - \varepsilon \\
\gamma(r/\varepsilon) &\text{for } r \in  (-\varepsilon, \varepsilon)\\
1 \quad &\text{for } r\geq \varepsilon
\end{cases} \, \, \; \text{with }
\gamma(t) \, = \, 2C\int_{-1}^t \mathrm{e}^{-1/(1-\tau^2)} \dtau - 1
\end{align}
and $C = (\int_{-1}^1 \mathrm{e}^{-1/(1-\tau^2)} \dtau)^{-1}$.
This function is a $C^\infty$ version of the regularized Tresca friction function that is found in \cite[Ex.\@ 5, pp.\@ 38]{Bre72}. In our example we choose $\varepsilon = 10^{-1}$.
Finite element simulations are carried out by the open source computing platform FEniCS (see \cite{AlnaesEtal14, Dolfinx23, BasixJoss, ScroggsEtal2022}) on a polygonal domain $\Omega_h \approx \Omega$.
The friction threshold $a$ on $\Gamma$ is defined as
\begin{align*}
a(x,y) \, = \, 2 + \mathrm{e}^{\sin(x^2y)}\quad \text{for } (x,y) \in \Gamma\, .
\end{align*}
We choose the right hand sides $f$ and $g$ in \eqref{eq:nonlin_robin} in such a way that the exact solution to \eqref{eq:nonlin_robin} is 
\begin{align*}
u(x,y) \, = \, xy \sin(5\pi\sqrt{x^2 + y^2}) \quad \text{for } (x,y)\in \Omega\, .
\end{align*}
This function changes its sign four times on the boundary $\Gamma$, so the nonlinear boundary condition is certainly active in the transition zone $(-\varepsilon, \varepsilon)$. 
We simulate the finite element approximation $u_h$ for a range of mesh sizes $h$ and consider the $L^2$ and $H^1$ errors.
The result is found in Figure~\ref{fig:numex1}. 
\begin{figure}[t!]
\centering 
\begin{minipage}[t]{0.44\textwidth}
\raisebox{3.5mm}{\includegraphics[scale=.363]{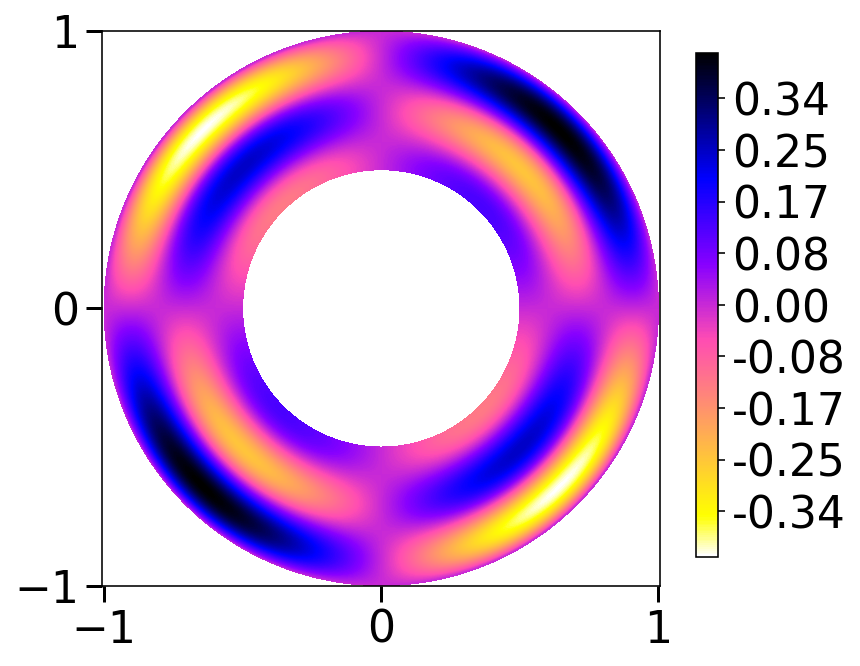}}
\end{minipage}
\begin{minipage}[t]{0.48\textwidth}
\hfill
\includegraphics[scale=.33]{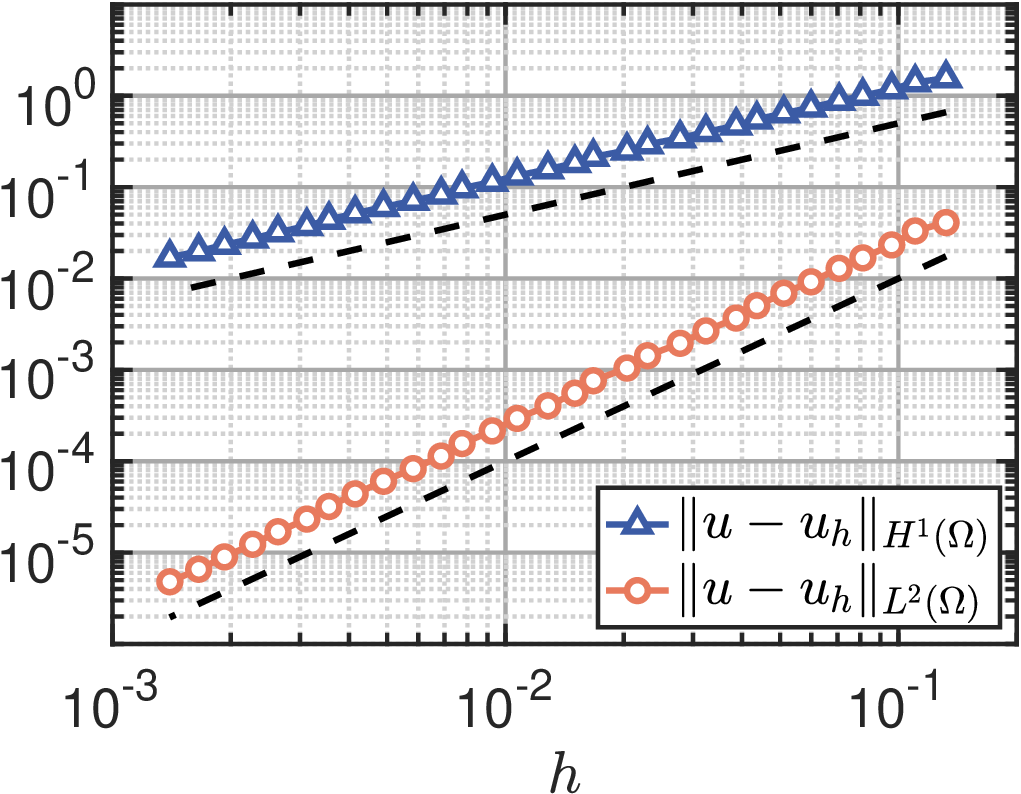}
\end{minipage}
\caption{
Left: The exact solution on the domain $\Omega_h$. 
Right: Convergence of the finite element solution as a function of the mesh size $h$. The $H^1$-error decreases linearly, the $L^2$-error decreases quadratically, in accordance with Lemma~\ref{lem:H1} and Lemma~\ref{lem:L2}. 
The dashed lines have slope one and two respectively.}
\label{fig:numex1}
\end{figure}
We see that the convergence orders are in accordance with Lemma~\ref{lem:H1} and Lemma~\ref{lem:L2}.
\section{Linearizations with respect to the friction threshold}
In this section we derive linearizations of the finite element approximations $u_h^{(a)}$ solving \eqref{eq:uh} and $z_h^{(a)}$ solving 
\begin{align}\label{eq:zh}
\dot{b}[u_h^{(a)}](z_h^{(a)},v) \, = \, \dot{\ell}_{1_\omega (u_h^{(a)} - q), 0}(v) \quad \text{for all } v \in V_{\diamond,h}\, 
\end{align}
with respect to $a$. 
Most of the upcoming estimates could be performed with the $L^\infty$ norm replacing the $C^2$ norm. Since we assume that $\tilde{a} \in C^2(\Gamma)$ we only write the latter one.
We use a similar framework as in \cite{BurKnoOks25}.
For a fixed $a\in C^2(\Gamma)$ let 
\begin{align}\label{def:Lambda}
U_a : D(U_a) \subset C^2(\Gamma) \to V_{\diamond,h}\, , \quad U_a(\eta) \, = \,  u_h^{(a+\eta)}\, ,
\end{align}
where $D(U_a)$ is a neighborhood of the zero function in $C^2(\Gamma)$ that is so small that $u_h^{(a+\eta)}$ is well-defined for any $\eta\in D(U_a)$.
Note that $U_a(0) = u_h^{(a)}$.
We are interested in characterizing the Fr\'echet derivative of the operator $U_a$ at zero, i.e., the operator $U_a'(0): C^2(\Gamma) \to H^1(\Omega)$ that satisfies
\begin{align*}
\frac{1}{\Vert \eta \Vert_{C^2(\Gamma)}} \Vert U_a(\eta) - U_a(0) - U_a'(0)\eta \Vert_{H^1(\Omega)} \, \to \, 0 \quad \text{as } \Vert \eta \Vert_{C^2(\Gamma)} \to 0\, . 
\end{align*}
We start with a result on the continuity of $U_a$.
\begin{lemma}\label{lem:continuity}
For $a \in C^2(\Gamma)$ with $0<a_0\leq a$ the operator $U_a$ from \eqref{def:Lambda} satisfies
\begin{align}\label{eq:conti}
\Vert U_a(\eta_1) - U_a(\eta_2)\Vert_{H^1(\Omega)}\, \lesssim \, \Vert \eta_1 - \eta_2 \Vert_{C^2(\Gamma)}
\end{align}
for any $\eta_1, \eta_2 \in D(U_a)$.
\end{lemma}
\begin{proof}
We take the difference of \eqref{eq:uh} with $a$ replaced by $a+\eta_j$ for $j=1,2$ and insert $v = u_h^{(a+\eta_1)} - u_h^{(a+\eta_2)}$ to obtain
\begin{align}\label{eq:comp11}
\begin{split}
0  &\,=\,  b(u_h^{(a+\eta_1)},v) - b(u_h^{(a+\eta_2)},v) \\
&\,=\,  \Vert \nabla v \Vert_{L^2(\Omega)}^2 + \int_{\Gamma}( (a+\eta_1) \beta(u_h^{(a+\eta_1)}) - (a+\eta_2) \beta(u_h^{(a+\eta_2)} ))v \ds \\
&\,=\,  \Vert \nabla v \Vert_{L^2(\Omega)}^2 + \int_{\Gamma}(a+\eta_2)(\beta(u_h^{(a+\eta_1)}) - \beta(u_h^{(a+\eta_2)}))v\ds \\
&\phantom{= \, }  + \int_{\Gamma} (\eta_1 - \eta_2) \beta(u_h^{(a+\eta_1)} )v \ds \, .
\end{split}
\end{align}
We note that the second term on the right hand side of \eqref{eq:comp11} is positive due to the assumption that $\beta$ is non-decreasing.
The Poincar\'e type inequality \eqref{eq:poincare-type-inequ} yields
\begin{align*}
\Vert v \Vert_{H^1(\Omega)}^2 \, \lesssim \, \Vert \eta_1 - \eta_2 \Vert_{C^2(\Gamma)} \Vert u_h^{(a+\eta_1)}  \Vert_{H^1(\Omega)} \Vert v \Vert_{H^1(\Omega)} \, .
\end{align*}
Together with \eqref{eq:ubound} this implies \eqref{eq:conti}.
\end{proof}
We return to the linearized Robin-type problem from \eqref{eq:uhdot} as $U_a'(0)\eta$ can be expressed in terms of its solution.
\begin{theorem}\label{thm:deruh}
Let $\dot{u}_h^{(a)} \in V_{\diamond,h}$ denote the unique solution of \eqref{eq:uhdot} with $\dot{f} = 0$ and $\dot{g} = -\eta \beta(u_h^{(a)})$,
where $u_h^{(a)}$ is defined by
\eqref{eq:uh}.
Then,
$\dot{u}_h^{(a)}= U_a'(0)\eta$.
\end{theorem}
\begin{proof}
We take the difference of \eqref{eq:uh} with $a$ replaced by $a+\eta_j$ for $j=1,2$ and additionally subtract \eqref{eq:uhdot} with $\dot{f}$ and $\dot{g}$ as given in the theorem. Then, we insert $v=u_h^{(a+\eta)} - u_h^{(a)} - \dot{u}_h^{(a)}$ as a test function to obtain
\begin{multline}\label{eq:frech}
\Vert \nabla v \Vert_{L^2(\Omega)}^2 + \int_{\Gamma} a ( \beta(u_h^{(a+\eta)}) - \beta(u_h^{(a)}) - \beta'(u_h^{(a)}) \dot{u}_h^{(a)})v \ds \\
\, = \, - \int_{\Gamma} \eta ( \beta(u_h^{(a+\eta)}) - \beta(u_h^{(a)}) ) v \ds \, .
\end{multline}
Using the Cauchy--Schwarz inequality, the differentiability of $\beta$ and the continuity result from Lemma \ref{lem:continuity} yields
\begin{equation}\label{eq:boundrhs}
\left| \int_{\Gamma} \eta ( \beta(u_h^{(a+\eta)}) - \beta(u_h^{(a)}) ) v \ds\right| \, \lesssim \, \Vert \eta \Vert_{C^2(\Gamma)}^2\Vert v \Vert_{H^1(\Omega)}\, .
\end{equation}
Furthermore, we write
\begin{multline}\label{eq:exp1}
a(\beta(u_h^{(a+\eta)}) - \beta(u_h^{(a)}) - \beta'(u_h^{(a)}) \dot{u}_h^{(a)})v \\
\, = \, a(\beta(u_h^{(a+\eta)}) - \beta(u_h^{(a)}) - \beta'(u_h^{(a)}) (u_h^{(a+\eta)} - u_h^{(a)} ))v + a\beta'(u_h^{(a)}) v^2
\end{multline}
and see that the second term on the right hand side of \eqref{eq:exp1} is positive.
Moreover, for the first term on the right hand side of \eqref{eq:exp1} we use that
\begin{align*}
|\beta(u) - \beta(v) - \beta'(v)(u-v)| \, \lesssim \, |u-v|^2 \quad \text{for any } u,v \in \R\, 
\end{align*}
and conclude by applying the Cauchy--Schwarz inequality, the Sobolev embedding $H^{1/2}(\Gamma) \subset L^4(\Gamma)$ and the Lipschitz continuity from \eqref{eq:conti} that
\begin{multline}\label{eq:dbbeta}
\left| \int_{\Gamma} a(\beta(u_h^{(a+\eta)}) - \beta(u_h^{(a)}) - \beta'(u_h^{(a)}) (u_h^{(a+\eta)} - u_h^{(a)} ))v \ds \right| \\
\, \lesssim \, \Vert a \Vert_{C^2(\Gamma)}\Vert \eta \Vert_{C^2(\Gamma)}^2 \Vert v \Vert_{H^1(\Omega)}\, ,
\end{multline}
Using \eqref{eq:boundrhs} and \eqref{eq:dbbeta} in \eqref{eq:frech} yields that
\begin{align*}
\Vert v \Vert_{H^1(\Omega)}^2 \, \lesssim \, \Vert \eta \Vert_{C^2(\Gamma)}^2 \Vert v \Vert_{H^1(\Omega)}\, ,
\end{align*}
which proves the claim.
\end{proof}

We apply the stability bound from Theorem~\ref{thm:stability} to $\dot{u}_h^{(a)}$ from Theorem \ref{thm:deruh}. This yields the next corollary.
\begin{corollary}\label{cor:stability}
Let $\eta \in V_J$. 
Then, the function $\dot{u}_h^{(a)} \in V_{\diamond,h}$ defined by \eqref{eq:uhdot} with $\dot{f} = 0$ and $\dot{g} = -\eta \beta(u_h^{(a)})$ satisfies
\begin{align}\label{eq:uhpbound}
\Vert \dot{u}_h^{(a)} \Vert_{H^1(\Omega)}\, \lesssim \, \Vert \dot{u}_h^{(a)} \Vert_{L^2(\omega)} + h \Vert \eta \Vert_{C^2(\Gamma)} \, .
\end{align}
The implicit constant in \eqref{eq:uhpbound} depends on $a$, on $\beta'$ and on the finite dimensional subspace $V_J$.
\end{corollary} 
\begin{proof}
We consider the linearized equation \eqref{eq:udot} and use $\tilde{f} = 0$ and $\tilde{g} = -\eta \beta(u^{(a)})$ in place of $\dot{f}$ and $\dot{g}$. 
Consider the finite dimensional space $\tilde{V}_J \subset H^{1/2}(\Gamma)$ spanned by $\phi_j \beta(u^{(a)})$ for $j=1,\dots,J$. By the stability result from Theorem \ref{thm:stability} we find that the corresponding solution to the linearized problem $\dot{u}^{(a)}\in H_\diamond^1(\Omega)$ satisfies
\begin{align}\label{eq:stabutilde}
\Vert \dot{u} ^{(a)}\Vert_{H^1(\Omega)} \, \lesssim \, \Vert \dot{u}^{(a)} \Vert_{L^2(\omega)} \, . 
\end{align}
The implicit constant in \eqref{eq:stabutilde} does not depend on $h$, since the weak formulation in \eqref{eq:udot} does not involve any $h$-dependence.
Let $\tilde{u}_h^{(a)} \in V_{\diamond,h}$ be the solution of \eqref{eq:uhdot} with $\tilde{f}$ and $\tilde{g}$.
The difference between $\tilde{u}_h^{(a)}$ and $\dot{u}_h^{(a)}$ is that the functions $\dot{g}$ and $\tilde{g}$ in the respective right hand side differ.
The finite element estimate from Lemma \ref{lem:felinearized} shows that
\begin{align}\label{eq:linfebound}
\Vert \dot{u}^{(a)} - \tilde{u}_h^{(a)} \Vert_{H^1(\Omega)} \, \lesssim \, h \Vert \dot{u}^{(a)} \Vert_{H^2(\Omega)}
\, &\lesssim \, h \Vert \eta \Vert_{C^2(\Gamma)} \Vert \beta(u^{(a)}) \Vert_{H^{1/2}(\Gamma)} \, .
\end{align}
Using \cite[Thm.\@ 6]{BoSi11} and a localization argument to $\Gamma$ as in the proof of Theorem \ref{thm:howp} further shows that
\begin{align*}
\Vert \beta(u^{(a)}) \Vert_{H^{1/2}(\Gamma)} \, \lesssim \, \Vert \beta' \Vert_{L^\infty(u(\Gamma))} \Vert u^{(a)} \Vert_{H^1(\Omega)}\, .
\end{align*}
We hide this bound in our implicit constant. 
The function $\tilde{u}_h^{(a)} - \dot{u}_h^{(a)}$ is a solution to \eqref{eq:uhdot} with $\dot{f}$ replaced by $0$ and $\dot{g}$ replaced by $\eta(\beta(u^{(a)}) - \beta(u_h^{(a)}))$.
Using the well-posedness bound, the differentiability of $\beta$ and Lemma~\ref{lem:H1} yields that
\begin{align}\label{eq:linfebound2}
\Vert \tilde{u}_h^{(a)} - \dot{u}_h^{(a)} \Vert_{H^1(\Omega)} \, \lesssim \, \Vert \eta(\beta(u^{(a)}) - \beta(u_h^{(a)})) \Vert_{L^2(\Gamma)} \, \lesssim \, h \Vert \eta \Vert_{C^2(\Gamma)} \, .
\end{align}
Combining \eqref{eq:linfebound} and \eqref{eq:linfebound2} shows that
\begin{align}\label{eq:error_to_full_uh}
\Vert \dot{u}^{(a)} - \dot{u}_h^{(a)} \Vert_{H^1(\Omega)} \, \lesssim \, h \Vert \eta \Vert_{C^2(\Gamma)} \, .
\end{align}
Finally we use the triangle inequality, \eqref{eq:stabutilde} and \eqref{eq:error_to_full_uh} to get that
\begin{align*}
\Vert \dot{u}_h^{(a)} \Vert_{H^1(\Omega)} \, &\lesssim \, \Vert \dot{u}^{(a)} \Vert_{H^1(\Omega)} + h \Vert \eta \Vert_{C^2(\Gamma)} \\
&\lesssim \,  \Vert \dot{u}^{(a)} \Vert_{L^2(\omega)} + h \Vert \eta \Vert_{C^2(\Gamma)} \\
&\lesssim \,  \Vert \dot{u}_h^{(a)} \Vert_{L^2(\omega)} + h \Vert \eta \Vert_{C^2(\Gamma)} \, ,
\end{align*}
which ends the proof.
\end{proof}
Next we study the problem \eqref{eq:c2}. 
Recall that $q = u^{(\tilde{a})}|_\omega$ and that $z^{(a)} \in H_\diamond^{1}(\Omega)$ denotes the solution to \eqref{eq:c2} with $u=u^{(a)}$.
In particular, if $a = \tilde{a}$, then $z^{(\tilde{a})}=0$.
The corresponding finite element approximation $z_h^{(\tilde{a})} \in V_{\diamond,h}$ is the solution to \eqref{eq:zh} with $a$ replaced by $\tilde{a}$.
The next lemma shows that $z_h^{(\tilde{a})}$ has an unexpectedly good $H^1$ convergence to zero.
\begin{lemma}\label{lem:nicezhconv}
Let $z_h^{(\tilde{a})} \in V_{\diamond,h}$ denote the solution to \eqref{eq:zh} with $a$ replaced by $\tilde{a}$. Then,
\begin{align*}
\Vert z_h^{(\tilde{a})} \Vert_{H^1(\Omega)} \, \lesssim \, h^2 \, .
\end{align*}
\begin{proof}
We combine the well-posedness result for linear elliptic partial differential equations with the $L^2$-bound for the nonlinear problem in Theorem \ref{lem:L2} and obtain that
\begin{align*}
\Vert z_h^{(\tilde{a})} \Vert_{H^1(\Omega)} \, \lesssim \, \Vert u_h^{(\tilde{a})} - q \Vert_{L^2(\omega)} \, \lesssim \, h^2 (\Vert u^{(\tilde{a})} \Vert_{H^2(\Omega)} + \Vert u^{(a)}\Vert_{H^2(\Omega)}^2) \, .
\end{align*}
\end{proof}
\end{lemma}
For proving Lipschitz continuity of finite element approximations satisfying \eqref{eq:zh} with respect to the parameter $a$, we consider a slightly more general result in the next lemma. 
We consider a Lipschitz continuity result for finite element approximations satisfying \eqref{eq:uhdot}.
This provides Lipschitz continuity not only for $z_h^{(a)}$, but also for $\dot{u}_h^{(a)}$ and for $\dot{z}_h^{(a)}$; a function that we introduce in Theorem~\ref{thm:derzh}.
\begin{lemma}\label{lem:gencont}
Let $\dot{f}^{(a)} \in (H^1(\Omega))'$ and $\dot{g}^{(a)} \in H^{-1/2}(\Gamma)$ be two families of distributions parametrized by $a \in C^2(\Gamma)$ and suppose the following continuity
\begin{subequations}\label{eq:fstargstarass}
\begin{align}
\Vert \dot{f}^{(a_1)} - \dot{f}^{(a_2)} \Vert_{(H^1(\Omega))'}   \, &\lesssim \, \Vert a_1 - a_2 \Vert_{C^2(\Gamma)} \, ,\\
 \Vert \dot{g}^{(a_1)} - \dot{g}^{(a_2)} \Vert_{H^{-1/2}(\Gamma) } \, &\lesssim \, \Vert a_1 - a_2 \Vert_{C^2(\Gamma)}\, . \label{eq:fstargstarass2}
\end{align}
\end{subequations}
Then, for $j=1,2$, the solutions $\dot{w}_h^{(a+\eta_j)} \in V_{\diamond, h}$ satisfying \eqref{eq:uhdot} with $a$, $\dot{f}$ and $\dot{g}$ replaced by $a+\eta_j$, $\dot{f}^{(a+\eta_j)}$ and $\dot{g}^{(a+\eta_j)}$ with $a\geq a_0>0$ and $\Vert \eta_j \Vert_{C^2(\Gamma)}$ sufficiently small satisfy
\begin{align*}
\Vert \dot{w}_h^{(a+\eta_1)} - \dot{w}_h^{(a+\eta_2)} \Vert_{H^1(\Omega)} \, \lesssim \, \Vert \eta_1 - \eta_2 \Vert_{C^2(\Gamma)}\, .
\end{align*}
\end{lemma}
\begin{proof}
The function $w^* = \dot{w}_h^{(a+\eta_1)} - \dot{w}_h^{(a+\eta_2)}$ is the unique solution to
\begin{align*}
\dot{b}[u_h^{(a+\eta_2)}](w^*,v) \, = \, \dot{\ell}_{\dot{f}^*, \dot{g}^* + \xi}(v) \quad \text{for all } v \in V_{\diamond, h} \, ,
\end{align*}
where
\begin{align*}
\dot{f}^* \, &= \, \dot{f}^{(a+ \eta_1)} - \dot{f}^{(a+\eta_2)} \, , \qquad
\dot{g}^* \, = \, \dot{g}^{(a+ \eta_1)} - \dot{g}^{(a+\eta_2)} \, ,\\
\xi \, &= \, \left( (a+\eta_2) \beta'(u_h^{(a+\eta_2)}) - (a+\eta_1) \beta'(u_h^{(a+\eta_1)}) \right) \dot{w}_h^{(a+\eta_1)} \, .
\end{align*}
We split $\xi$ into $\xi = \xi_1 + \xi_2$ with
\begin{subequations}\label{def:xi1xi2}
\begin{align}
\xi_1 \, &= \, (a+\eta_2) \left( \beta'(u_h^{(a+\eta_2)}) - \beta'(u_h^{(a+\eta_1)})\right) \dot{w}_h^{(a+\eta_1)} \, , \\
\xi_2 \, &= \, (\eta_2 - \eta_1) \beta'(u_h^{(a+\eta_1)})\dot{w}_h^{(a+\eta_1)}\, .
\end{align}
\end{subequations}
The second term $\xi_2$ can be easily bounded since $\beta' \in C_b^2(\R)$ to get
\begin{align*}
\Vert \xi_2 \Vert_{L^2(\Gamma)} \, \lesssim \, \Vert \eta_1 - \eta_2 \Vert_{L^\infty(\Gamma)}\Vert \dot{w}_h^{(a+\eta_1)} \Vert_{H^1(\Omega)} \, .
\end{align*}
For the first term we use H\"older's inequality with coefficients that satisfy $1/4 + 1/4 = 1/2$, the Lipschitz continuity of $\beta'$, the Sobolev embedding $H^{1/2}(\Gamma) \subset L^4(\Gamma)$ (see, e.g., \cite[Thm.\@ 6.7]{DiNPaVal12}) and Lemma~\ref{lem:continuity} to get that
\begin{align*}
\Vert \xi_1 \Vert_{L^2(\Gamma)} \, &\lesssim \, \left\Vert \beta'(u_h^{(a+\eta_2)}) - \beta'(u_h^{(a+\eta_1)}) \right\Vert_{L^4(\Gamma)} \Vert \dot{w}_h^{(a+\eta_1)} \Vert_{L^4(\Gamma)} \\
\, &\lesssim \, \left\Vert u_h^{(a+\eta_2)} - u_h^{(a+\eta_1)} \right\Vert_{H^{1/2}(\Gamma)} \Vert \dot{w}_h^{(a+\eta_1)} \Vert_{H^{1/2}(\Gamma)} \\
\, &\lesssim \, \Vert \eta_1 - \eta_2 \Vert_{C^2(\Gamma)} \Vert \dot{w}_h^{(a+\eta_1)} \Vert_{H^1(\Omega)} \, .
\end{align*}
We remark that $\Vert \dot{w}_h^{(a+\eta_1)} \Vert_{H^1(\Omega)}$ can be further estimated by the right hand sides in its corresponding weak formulation. In total, we get that
\begin{align}\label{eq:xibound}
\Vert \xi \Vert_{L^2(\Gamma)} \, \lesssim \, \Vert \eta_1 - \eta_2 \Vert_{C^2(\Gamma)} \, .
\end{align}
The well-posedness bound for $w^*$ now yields that
\begin{align*}
\Vert w^* \Vert_{H^1(\Omega)} \, \lesssim \, \Vert \dot{f}^* \Vert_{H^1(\Omega)'} + \Vert \dot{g}^* \Vert_{H^{-1/2}(\Gamma)} + \Vert \xi \Vert_{L^2(\Gamma)}
\end{align*}
with an implicit constant that does not depend on the mesh size $h$. The assumptions in \eqref{eq:fstargstarass} together with \eqref{eq:xibound} conclude the proof.
\end{proof}
We use the same terminology for the linearization of $z_h^{(a)}$ as for $u_h^{(a)}$ in \eqref{def:Lambda}.
We consider the definition for $z_h^{(a)} \in V_{\diamond,h}$ from \eqref{eq:zh} and for a fixed $a\in C^2(\Gamma)$ we introduce
\begin{align}\label{def:Lambda2}
Z_a : D(Z_a) \subset C^2(\Gamma) \to V_{\diamond,h}\, , \quad Z_a(\eta) \, = \,  z_h^{(a+\eta)}\, ,
\end{align}
where $D(Z_a)$ is a neighborhood of the zero function in $C^2(\Gamma)$ that is so small that $z_h^{(a+\eta)}$ is well-defined for any $\eta\in D(Z_a)$.
We also note that $Z_a(0) = z_h^{(a)}$.
The following corollary is a consequence of Lemma~\ref{lem:continuity} together with Lemma~\ref{lem:gencont}.
\begin{corollary}\label{cor:zhlin}
For $a \in C^2(\Gamma)$ with $0<a_0\leq a$ the operator $Z_a$ from \eqref{def:Lambda2} satisfies
\begin{align*}
\Vert Z_a(\eta_1) - Z_a(\eta_2)\Vert_{H^1(\Omega)}\, \lesssim \, \Vert \eta_1 - \eta_2 \Vert_{C^2(\Gamma)}
\end{align*}
for any $\eta_1, \eta_2 \in D(Z_a)$.
\end{corollary}
Next, we study the linearization of $Z_a$ at zero. Again, we find a characterization of $Z_a'(0)$ in terms of a finite element approximation to a Robin-type problem.
\begin{theorem}\label{thm:derzh}
Let $\dot{z}_h^{(a)} \in V_{\diamond,h}$ denote the unique solution of \eqref{eq:uhdot} with 
\begin{align}\label{eq:dotfdotgdotzh}
\dot{f} \, = \,  1_\omega\dot{u}_h^{(a)} \, , \quad \text{and } \quad
\dot{g} \, = \,  -\eta \beta'(u_h^{(a)})z_h^{(a)} - a \beta''(u_h^{(a)}) \dot{u}_h^{(a)}z_h^{(a)} \, ,
\end{align}
where $u_h^{(a)}$ and $z_h^{(a)}$ are defined by
\eqref{eq:uh} and \eqref{eq:zh} and $\dot{u}_h^{(a)}$ is defined in Theorem~\ref{thm:deruh}.
Then,
$\dot{z}_h^{(a)}= Z_a'(0)\eta$.
\end{theorem}
\begin{proof}
We use notation introduced in the proof of Lemma~\ref{lem:gencont}.
Subtracting the weak formulations for $z_h^{(a+\eta)}$ and $z_h^{(a)}$ from \eqref{eq:zh} yields that $z^*=z_h^{(a+\eta)} - z_h^{(a)}$ solves
\begin{align*}
\dot{b}[u_h^{(a)}](z^*, v) \, = \, \dot{\ell}_{1_\omega(u_h^{(a+\eta)} - u_h^{(a)}), \xi}(v) \quad \text{for all } v \in V_{\diamond, h}\, ,
\end{align*}
where $\xi = \xi_1 + \xi_2$ is given just as in \eqref{def:xi1xi2} with the replacements $\eta_1 = \eta$, $\eta_2 = 0$ and $\dot{w}_h = \dot{z}_h$.
Subtracting now the weak formulation for $\dot{z}_h^{(a)}$
yields that $z^* - \dot{z}_h^{(a)}$ satisfies
\begin{align*}
\dot{b}[u_h^{(a)}](z^*- \dot{z}_h^{(a)}, v) \, = \, \dot{\ell}_{\dot{\alpha}, \dot{\beta}}(v) \quad \text{for all } v \in V_{\diamond, h}\, 
\end{align*}
with $\dot{\beta} = \dot{\beta}_1 + \dot{\beta}_2$ and
\begin{align*}
\dot{\alpha} \, &= \, 1_\omega(u_h^{(a+\eta)} - u_h^{(a)} - \dot{u}_h^{(a)}) \, , \\
\dot{\beta}_1 \, &= \, a \left( \beta'(u_h^{(a)}) - \beta'(u_h^{(a+\eta)}) \right) z_h^{(a+\eta)} + a \beta''(u_h^{(a)}) \dot{u}_h^{(a)} z_h^{(a)} \, , \\
\dot{\beta}_2 \, &= \, -\eta \beta'(u_h^{(a+\eta)}) z_h^{(a+\eta)} + \eta \beta'(u_h^{(a)}) z_h^{(a)} \, .
\end{align*}
Due to the well-posedness bound we have that
\begin{align}\label{eq:frechbound}
\Vert z^*- \dot{z}_h^{(a)} \Vert_{H^1(\Omega)} \, \lesssim \, \Vert \dot{\alpha} \Vert_{(H^1(\Omega))'} + \Vert \dot{\beta} \Vert_{H^{-1/2}(\Gamma)} \, .
\end{align}
All we need to show is that the right hand side of \eqref{eq:frechbound} decays as $\Vert \eta \Vert_{C^2(\Gamma)}^2$ as $\Vert \eta \Vert_{C^2(\Gamma)} \to 0$.
By Theorem~\ref{thm:deruh} we immediately obtain that
\begin{align*}
\Vert \dot{\alpha} \Vert_{(H^1(\Omega))'} \, \lesssim \, \Vert \eta \Vert_{C^2(\Gamma)}^2\, .
\end{align*}
For the bound of $\dot{\beta}$ we further split $\dot{\beta}_1=\dot{\beta}_{1,1} + \dot{\beta}_{1,2}$ and $\dot{\beta}_2=\dot{\beta}_{2,1} + \dot{\beta}_{2,2}$ with
\begin{align*}
\dot{\beta}_{1,1} \, &= \, a \left(\beta'(u_h^{(a)}) - \beta'(u_h^{(a+\eta)}) + \beta''(u_h^{(a)})\dot{u}_h^{(a)} \right) z_h^{(a+\eta)} \, , \\
\dot{\beta}_{1,2} \, &= \, a \beta''(u_h^{(a)}) \dot{u}_h^{(a)} ( z_h^{(a)} - z_h^{(a+\eta)})\, , \\
\dot{\beta}_{2,1} \, &= \, -\eta \left( \beta'(u_h^{(a+\eta)}) - \beta'(u_h^{(a)}) \right) z_h^{(a+\eta)} \, ,\\
\dot{\beta}_{2,2} \, &= \, \eta \beta'(u_h^{(a)}) (z_h^{(a)} - z_h^{(a+\eta)}) \, .
\end{align*}
Bounding $\dot{\beta}_{2,2}$ is simple and can be done by using Corollary~\ref{cor:zhlin} to get
\begin{align*}
\Vert \dot{\beta}_{2,2} \Vert_{H^{-1/2}(\Gamma)} \, \lesssim \, \Vert \dot{\beta}_{2,2} \Vert_{L^2(\Gamma)} \, \lesssim \, \Vert \eta \Vert_{C^2(\Gamma)}^2 \, .
\end{align*}
For $\dot{\beta}_{2,1}$ we use Hölder's inequality with coefficients satisfying $1/4+1/4=1/2$, the embedding $H^{1/2}(\Gamma)$ into $L^4(\Gamma)$ and Lemma~\ref{lem:continuity} to get
\begin{align*}
\Vert \dot{\beta}_{2,1} \Vert_{H^{-1/2}(\Gamma)} \, &\lesssim \, \Vert \dot{\beta}_{2,1} \Vert_{L^2(\Gamma)} \\
\, &\lesssim \, \Vert \eta \Vert_{C^2(\Gamma)} \Vert \beta'(u_h^{(a+\eta)}) - \beta'(u_h^{(a)}) \Vert_{L^4(\Gamma)}  \Vert z_h^{(a+\eta)} \Vert_{L^4(\Gamma)} \\
\, &\lesssim \, \Vert \eta \Vert_{C^2(\Gamma)} \Vert u_h^{(a+\eta)} - u_h^{(a)} \Vert_{L^4(\Gamma)}  \Vert z_h^{(a+\eta)} \Vert_{L^4(\Gamma)} \\
\, &\lesssim \, \Vert \eta \Vert_{C^2(\Gamma)}^2 \, .
\end{align*}
For $\dot{\beta}_{1,2}$ we use similar estimates and note that due to the well-posedness estimate for $\dot{u}_h^{(a)}$ and Corollary~\ref{cor:zhlin} we get that
\begin{align*}
\Vert \dot{\beta}_{1,2} \Vert_{H^{-1/2}(\Gamma)} \, &\lesssim \, \Vert \dot{u}_h^{(a)} \Vert_{H^1(\Omega)} \Vert z_h^{(a)} - z_h^{(a+\eta)} \Vert_{H^1(\Omega)} 
\, \lesssim \, \Vert \eta \Vert_{C^2(\Gamma)}^2 \, .
\end{align*}
The bound for $\dot{\beta}_{1,1}$ is the most delicate one. First we split $\dot{\beta}_{1,1} = \dot{\beta}_{1,1,1} + \dot{\beta}_{1,1,2}$ with
\begin{align*}
\dot{\beta}_{1,1,1} \, &= \, -a\left( \beta'(u_h^{(a+\eta)}) - \beta'(u_h^{(a)}) - \beta''(u_h^{(a)}) (u_h^{(a+\eta)} - u_h^{(a)}) \right) z_h^{(a+\eta)} \, , \\
\dot{\beta}_{1,1,2} \, &= \, -a\beta''(u_h^{(a)}) ( u_h^{(a+\eta)} - u_h^{(a)} - \dot{u}_h^{(a)} ) z_h^{(a+\eta)} \, .
\end{align*}
As before, $\dot{\beta}_{1,1,2}$ can be bounded by using Hölder's inequality applied just like before together with Theorem~\ref{thm:deruh} to get that
\begin{align*}
\Vert \dot{\beta}_{1,1,2} \Vert_{H^{-1/2}(\Gamma)}\, \lesssim \, \Vert \eta \Vert_{C^2(\Gamma)}^2 \, .
\end{align*}
To bound $\dot{\beta}_{1,1,1}$ we first note that 
\begin{align}\label{eq:betaprimederprop}
|\beta'(u) - \beta'(v) - \beta''(v)(u-v)| \, \lesssim \, |u-v|^2 \quad \text{for any } u,v \in \R\, .
\end{align}
We now use in succession (i) $L^{4/3} \subset H^{-1/2}(\Gamma)$ (since $H^{1/2}(\Gamma) \subset L^4(\Gamma)$), (ii) the inequality \eqref{eq:betaprimederprop}, (iii) Hölder's inequality with coefficients satisfying $	1/4+1/4+1/4=3/4$, (iv) the embedding $H^{1/2}(\Gamma) \subset L^4(\Gamma)$ and (v) Lemma \ref{lem:continuity} to see that
\begin{align}\label{eq:compcrit}
\begin{split}
\Vert \dot{\beta}_{1,1,1} \Vert_{H^{-1/2}(\Gamma)}\, &\lesssim \, \Vert \dot{\beta}_{1,1,1} \Vert_{L^{4/3}(\Gamma)} \\
& \lesssim \, \Vert (u_h^{(a+\eta)} - u^{(a)})^2   z_h^{(a+\eta)} \Vert_{L^{4/3}(\Gamma)} \\
& \lesssim \, \Vert u_h^{(a+\eta)} - u^{(a)} \Vert_{L^4(\Gamma)}^2 \Vert z_h^{(a+\eta)} \Vert_{L^4(\Gamma)} \\
& \lesssim \, \Vert u_h^{(a+\eta)} - u^{(a)} \Vert_{H^{1/2}(\Gamma)}^2 \Vert z_h^{(a+\eta)} \Vert_{H^{1/2}(\Gamma)} \\
& \lesssim \, \Vert \eta \Vert_{C^2(\Gamma)}^2 \, .
\end{split}
\end{align}
This finishes the proof.
\end{proof}

\begin{remark}\label{rmk:cont}
Using the general continuity result from Lemma~\ref{lem:gencont} one can see that both $\dot{u}_h^{(a)}$ and $\dot{z}_h^{(a)}$ from Theorem~\ref{thm:deruh} and Theorem~\ref{thm:derzh} are Lipschitz continuous with respect to the parameter $a$.
The Lipschitz constants do not depend on the mesh size $h$.
To see this, one needs to verify that the conditions \eqref{eq:fstargstarass} are fulfilled.
For $\dot{u}_h^{(a)}$ this is rather straightforward, for $\dot{z}_h^{(a)}$ however, we need to apply similar steps as in the computation \eqref{eq:compcrit}.
We do not perform this computation in detail. We emphasize, however, that this analysis requires $\beta'''$ to exist. For $\dot{z}_h^{(a)}$ from Theorem~\ref{thm:derzh} to be Lipschitz continuous with respect to $a$, one needs to check that \eqref{eq:fstargstarass} for \eqref{eq:dotfdotgdotzh} is fulfilled. In order to prove \eqref{eq:fstargstarass2}, proceeding analogously as before, the difference $\beta''(u_h^{(a+\eta_1)}) - \beta''(u_h^{(a+\eta_2)})$ can be bounded by using Lipschitz continuity of $\beta''$ and Lemma~\ref{lem:continuity}.
\end{remark}

\section{Stable recovery of the friction parameter}
In \eqref{eq:Fh} we defined the function $F_h : V_J \to \R^J$ that needs to vanish in order for the Lagrangian $\Theta$ from \eqref{eq:theta} to have a saddle point. 
By now we can be more precise with the functions that appear in its definition.
The finite element approximations $u_h \in V_{\diamond,h}$ and $z_h \in V_{\diamond,h}$ are defined in \eqref{eq:uh} and \eqref{eq:zh}.
We can also determine the Fr\'echet derivative of $F_h$. For a fixed $a \in C^2(\Gamma)$ we define it by $\dot{F}_h[a]: V_J \to \R^J$, where
\begin{subequations}\label{eq:Fhdot}
\begin{align}
\dot{F}_h[a]\eta \, &= \, [\dot{F}_{h,1}[a]\eta, \dots, \dot{F}_{h,J}[a]\eta]\, , \quad \text{with } \\
 \dot{F}_{h,j}[a]\eta \, &= \, \int_{\Gamma} \phi_j \big(\beta'(u_h^{(a)})\dot{u}_h^{(a)} z_h^{(a)} + \beta(u_h^{(a)}) \dot{z}_h^{(a)} \big) \ds\, \quad \text{for } j=1,\dots, J\, .
\end{align}
\end{subequations}
The finite element approximations $\dot{u}_h^{(a)}$ and $\dot{z}_h^{(a)}$ are defined in Theorem~\ref{thm:deruh} and Theorem~\ref{thm:derzh}, respectively.
For the function $\dot{F}_h[a]$ we can immediately conclude two corollaries that follow from our previous analysis.
\begin{corollary}\label{cor:lipFdot}
The Fr\'echet derivative $\dot{F}_h[a]$ is locally Lipschitz continuous in $a \in V_J$, i.e., there is a $\rho>0$ such that for any $a_1, a_2 \in B_\rho(a)$ there holds
\begin{align}\label{eq:frechFhdot}
\Vert \dot{F}_h[a_1] - \dot{F}_h[a_2] \Vert_{\R^J \leftarrow V_J} \, \lesssim \, \Vert a_1 - a_2 \Vert_{C^2(\Gamma)}\, .
\end{align}
The constant in \eqref{eq:frechFhdot} does not depend on the mesh size $h$.
\end{corollary}
\begin{proof}
By Lemma~\ref{lem:continuity}, Corollary~\ref{cor:zhlin} and Remark~\ref{rmk:cont} the functions $u_h^{(a)}, z_h^{(a)}, \dot{u}_h^{(a)}$ and $\dot{z}_h^{(a)}$ are locally Lipschitz continuous in the parameter $a \in V_J$. Since $\dot{F}_{h,j}[a]\eta$ consists of products and a sum of these functions, the result follows. 
\end{proof}

\begin{corollary}\label{cor:consistency}
For the function $F_h(\tilde{a})$ from \eqref{eq:Fh} there holds
\begin{align*}
| F_h(\tilde{a})| \, \lesssim \, h^2\, \quad \text{as } h \to 0\, .
\end{align*}
\end{corollary}
\begin{proof}
This is a consequence of the Cauchy--Schwarz inequality and the convergence property from Lemma~\ref{lem:nicezhconv}.
Note that this requires that $F$ is evaluated at $\tilde{a}$.
\end{proof}
The framework established in \cite{Keller75} requires three conditions to be satisfied for the Newton scheme applied to $F_h$ to become a stable recovery method for $\tilde{a}$. 
These conditions are (i) the Lipschitz continuity of $\dot{F}_h$ from Corollary~\ref{cor:lipFdot}, (ii) the consistency of $F_h$ of order $p$ (here $p=2$) from Corollary~\ref{cor:consistency} and (iii) uniform boundedness of $(\dot{F}_h[\tilde{a}])^{-1}$. The latter property will be proven in the next lemma leading to our main result.

\begin{lemma}\label{lem:boundedinv}
Suppose that $\beta(u^{(\tilde{a})})$ vanishes at most at isolated points on $\Gamma$. 
For a sufficiently small maximal mesh size $h$ the Fr\'echet derivative $\dot{F}_h[\tilde{a}]$ of $F_h$ at $\tilde{a}$ satisfies
\begin{align*}
\Vert \eta \Vert_{C^2(\Gamma)} \, \lesssim \, | \dot{F}_h[\tilde{a}] \eta |_{\R^J} \, \quad \text{for all } \eta \in V_J\, .
\end{align*}
This implies that the inverse operator $(\dot{F}_h[\tilde{a}])^{-1}$ exists and that it is uniformly bounded independently of the mesh size $0<h<h_0$ for some $h_0>0$.
\end{lemma}
\begin{proof}
Any function $\eta \in V_J$ can be written as a linear combination of the functions $\phi_1, \dots, \phi_J$.
If $\eta$ has the representation
\begin{align*}
\eta \, = \, \sum_{j=1}^J \eta_j \phi_j \quad \text{then, for the vector } \; \underline{\eta} \, = \, [\eta_1, \dots, \eta_J]^T \in \R^J
\end{align*}
we obtain the norm equivalence $\Vert \eta \Vert_{C^2(\Gamma)} \sim | \underline{\eta} |_{\R^J}$. Moreover, by using the representation of $\dot{F}_h[a]$ from \eqref{eq:Fhdot} we find that
\begin{align}\label{eq:Fhdoteta}
\underline{\eta} \cdot \dot{F}_h[a] \eta \, = \, \int_\Gamma \eta \big(\beta'(u_h^{(a)})\dot{u}_h^{(a)} z_h^{(a)} + \beta(u_h^{(a)}) \dot{z}_h^{(a)} \big) \ds \, .
\end{align}
Following the proof of \cite[Lem.\@ 4.4]{BurKnoOks25} it can be seen that $a \mapsto \Vert a \beta(u^{(\tilde{a})})\Vert_{H^{-1/2}(\Gamma)}$ is a norm on $C^2(\Gamma)$.
This uses the assumption that $\beta(u^{(\tilde{a})})$ does not vanish on an open subset of $\Gamma$. Since $V_J \subset C^2(\Gamma)$ is a finite dimensional subspace we have the norm equivalence
\begin{align*}
\Vert \eta \Vert_{C^2(\Gamma)} \, \sim \, \Vert \eta \beta(u^{(\tilde{a})})\Vert_{H^{-1/2}(\Gamma)} \quad \text{for } \eta \in V_J\, .
\end{align*}
The constants in the norm equivalence estimates do not depend on the mesh size $h$ (but do depend of course on the dimension $J$).
As in the proof of Corollary~\ref{cor:stability} we consider the solution $\dot{u}^{(\tilde{a})}$ of \eqref{eq:udot} with $\tilde{f} = 0$ and $\tilde{g} = -\eta \beta(u^{(\tilde{a})})$ in place of $\dot{f}$ and $\dot{g}$.
We recall that $\dot{u}^{(\tilde{a})}$ satisfies the boundary condition \eqref{eq:linRob2} on $\Gamma$.
Accordingly, we use the norm equivalence, the boundary condition on $\Gamma$, the bound of $\partial_\nu \dot{u}^{(\tilde{a})}$ from, e.g., \cite[Lem.\@ 4.3]{McLean00}, the $H^1$ bound from Lemma~\ref{lem:felinearized}
and the stability bound for $\dot{u}_h^{(\tilde{a})}$ from Corollary~\ref{cor:stability} to conclude that 
\begin{align*}
\Vert \eta \Vert_{C^2(\Gamma)} \, \lesssim \, \Vert \eta \beta(u^{(\tilde{a})})\Vert_{H^{-1/2}(\Gamma)} \, &= \, \Vert \partial_\nu \dot{u}^{(\tilde{a})} + \tilde{a} \beta'(u^{(\tilde{a})})\dot{u}^{(\tilde{a})} \Vert_{H^{-1/2}(\Gamma)} \\
& \lesssim \, \Vert \partial_\nu \dot{u}^{(\tilde{a})} \Vert_{H^{-1/2}(\Gamma)} + \Vert \tilde{a} \beta'(u^{(\tilde{a})})\dot{u}^{(\tilde{a})}\Vert_{H^{-1/2}(\Gamma)} \\
&\lesssim \, \Vert \dot{u}^{(\tilde{a})} \Vert_{H^1(\Omega)} \\
&\lesssim \, \Vert \dot{u}_h^{(\tilde{a})} \Vert_{H^1(\Omega)} + h \Vert \eta \Vert_{C^2(\Gamma)} \\
&\lesssim \, \Vert \dot{u}_h^{(\tilde{a})} \Vert_{L^2(\omega)} + h \Vert \eta \Vert_{C^2(\Gamma)} \, .
\end{align*}
For a sufficiently small maximal mesh size $h$ this inequality can be rearranged to obtain
\begin{align}\label{eq:etabound}
\Vert \eta \Vert_{C^2(\Gamma)} \, \lesssim \,\Vert \dot{u}_h^{(\tilde{a})} \Vert_{L^2(\omega)}  \, .
\end{align}
We return to the weak formulation of $\dot{z}_h^{(\tilde{a})}$ from Theorem~\ref{thm:derzh} and insert $\dot{u}_h^{(\tilde{a})}$ as a test function. This yields
\begin{multline}\label{eq:zhdotwuhdot}
\Vert \dot{u}_h^{(\tilde{a})} \Vert_{L^2(\omega)}^2 \, = \, - \int_\Omega \nabla \dot{z}_h^{(\tilde{a})} \cdot \nabla \dot{u}_h^{(\tilde{a})} \dx - \int_\Gamma a \beta'({u}_h^{(\tilde{a})}) \dot{z}_h^{(\tilde{a})}\dot{u}_h^{(\tilde{a})} \ds \\
-\int_\Gamma \big( \eta \beta'(u_h^{(\tilde{a})}) + a \beta''(u_h^{(\tilde{a})}) \dot{u}_h^{(\tilde{a})}\big) z_h^{(\tilde{a})} \dot{u}_h^{(\tilde{a})} \ds\, .
\end{multline}
Using the weak formulation for $\dot{u}_h^{(\tilde{a})}$ from Theorem~\ref{thm:deruh} and inserting $\dot{z}_h^{(\tilde{a})}$ as a test function one realizes that
\begin{align*}
- \int_\Omega \nabla \dot{z}_h^{(\tilde{a})} \cdot \nabla \dot{u}_h^{(\tilde{a})} \dx - \int_\Gamma a \beta'({u}_h^{(\tilde{a})}) \dot{z}_h^{(\tilde{a})}\dot{u}_h^{(\tilde{a})} \ds \, = \, \int_\Gamma \eta \beta(u_h^{(\tilde{a})}) \dot{z}_h^{(\tilde{a})} \ds \, .
\end{align*}
Therefore, recalling \eqref{eq:Fhdoteta}, the equation \eqref{eq:zhdotwuhdot} turns into
\begin{align}\label{eq:uhdotl2}
\Vert \dot{u}_h^{(\tilde{a})} \Vert_{L^2(\omega)}^2 \, = \, \underline{\eta} \cdot \dot{F}_h[\tilde{a}]\eta + T_1 + T_2 \, ,
\end{align}
with
\begin{align*}
T_1 \, = \, -2 \int_\Gamma \eta \beta'(u_h^{(\tilde{a})}) z_h^{(\tilde{a})} \dot{u}_h^{(\tilde{a})} \ds \,  \quad \text{and } \quad T_2 \, = \, - \int_\Gamma a \beta''(u_h^{(\tilde{a})}) (\dot{u}_h^{(\tilde{a})})^2 z_h^{(\tilde{a})} \ds\, .
\end{align*}
To find a bound for $T_1$ we use H\"older's inequality and the Sobolev embedding $H^{1/2}(\Gamma) \subset L^4(\Gamma)$ similarly as in the previous proofs and find that
\begin{align*}
|T_1| \, &\lesssim \, \Vert \eta \Vert_{C^2(\Gamma)} \Vert \beta'(u_h^{(\tilde{a})}) \Vert_{L^2(\Gamma)} \Vert z_h^{(\tilde{a})} \Vert_{L^4(\Gamma)} \Vert \dot{u}_h^{(\tilde{a})} \Vert_{L^4(\Gamma)} \\
& \lesssim \, \Vert \eta \Vert_{C^2(\Gamma)} \Vert u^{(\tilde{a})} \Vert_{H^2(\Omega)} (1+h) \Vert z_h^{(\tilde{a})} \Vert_{H^1(\Omega)} \Vert \dot{u}_h^{(\tilde{a})} \Vert_{H^1(\Omega)} \, .
\end{align*}
Now we use Lemma~\ref{lem:nicezhconv} and Corollary~\ref{cor:stability} together with the basic inequality $ab \leq (a^2 + b^2)/2$ and get that
\begin{align}\label{eq:T1bound}
\begin{split}
|T_1| \, \lesssim \, (h \Vert \eta \Vert_{C^2(\Gamma)}) ( h \Vert \dot{u}_h^{(\tilde{a})} \Vert_{H^1(\Omega)} )
\, &\lesssim \, h^2 \Vert \eta \Vert_{C^2(\Gamma)}^2 + h^2 \Vert \dot{u}_h^{(\tilde{a})} \Vert_{H^1(\Omega)}^2 \\
\, &\lesssim \, h^2 \Vert \eta \Vert_{C^2(\Gamma)}^2 + h^2 \Vert \dot{u}_h^{(\tilde{a})} \Vert_{L^2(\omega)}^2 \, .
\end{split}
\end{align}
Using the same arguments for $T_2$ yields
\begin{align}\label{eq:T2bound}
\begin{split}
|T_2| \, &\lesssim \, \Vert a \Vert_{C^2(\Gamma)} \Vert \beta''(u_h^{(\tilde{a})}) \Vert_{L^4(\Gamma)} \Vert z_h^{(\tilde{a})} \Vert_{L^4(\Gamma)} \Vert \dot{u}_h^{(\tilde{a})} \Vert_{L^4(\Gamma)}^2 \, \\ &\lesssim \, h^2 \Vert \dot{u}_h^{(\tilde{a})} \Vert_{L^2(\omega)}^2 + h^4 \Vert \eta \Vert_{C^2(\Gamma)}^2\, .
\end{split}
\end{align}
Returning to \eqref{eq:uhdotl2} and using \eqref{eq:T1bound} and \eqref{eq:T2bound} yields after a rearrangement that
\begin{align*}
\Vert \dot{u}_h^{(\tilde{a})} \Vert_{L^2(\omega)}^2 \, \lesssim \, |\underline{\eta} \cdot \dot{F}_h[\tilde{a}]\eta| + h^2 \Vert \eta \Vert_{C^2(\Gamma)}^2 \, ,
\end{align*}
Another rearrangement in combination with \eqref{eq:etabound} now shows that
\begin{align*}
\Vert \eta \Vert_{C^2(\Gamma)}^2 \, \lesssim \, |\underline{\eta} \cdot \dot{F}_h[\tilde{a}]\eta|\, .
\end{align*}
Finally, the Cauchy--Schwarz inequality in $\R^J$ together with the norm equivalence $\Vert \eta \Vert_{C^2(\Gamma)} \sim | \underline{\eta} |_{\R^J}$ in $V_J$ concludes the proof.
\end{proof}
We can now state our main result for the reconstruction of $a$, which follows from \cite[Thm.\@ 3.6, Thm.\@ 3.7]{Keller75}.
We formulate it for potentially noisy data $q^\delta = u|_\omega + \delta$, where $\delta \in L^2(\omega)$.
To be precise, the previously proven Corollary~\ref{cor:lipFdot}, Corollary~\ref{cor:consistency} and Lemma~\ref{lem:boundedinv} only provide the result for $\delta = 0$.
For general $\delta \in L^2(\omega)$ with sufficiently small norm $\Vert \delta \Vert_{L^2(\omega)}$ the theorem can be proven by small amendments of the previous results similarly, as in \cite[Sec.\@ 5]{BurKnoOks25}.
To fix the notation for the noisy case, let us denote by $z_h^{(a), \delta} \in V_{\diamond,h}$ the solution to \eqref{eq:zh}, when $q$ is replaced by $q^\delta$. Moreover, denote the corresponding linearization from Theorem~\ref{thm:derzh} by $\dot{z}_h^{(a),\delta} \in V_{\diamond,h}$.
The functions $F_h^\delta$ and $\dot{F}_h^\delta[a]$ that replace \eqref{eq:Fh} and \eqref{eq:Fhdot} in the noisy case arise by replacing $z_h^{(a)}$ and $\dot{z}_h^{(a)}$ by $z_h^{(a),\delta}$ and $\dot{z}_h^{(a),\delta}$.
If there is no noise, i.e., $\delta=0$, the operators $F_h^\delta$ and $\dot{F}_h^\delta[a]$ coincide with $F_h$ and $\dot{F}_h[a]$.
\begin{theorem}\label{thm:main}
For given $f \in H^1(\Omega)$ and $g \in H^{3/2}(\Gamma)$, let $u^{(\tilde{a})} \in H_\diamond^1(\Omega)$ denote the solution to \eqref{eq:nonlin_robin} with $a = \tilde{a}$. Assume that $\beta(u^{(\tilde{a})})$ does not vanish on an open subset of $\Gamma$.
Let $\omega \subset \Omega$ be open and write $q = u^{(\tilde{a})}|_\omega$.
Moreover, let $\delta \in L^2(\omega)$ be a perturbation with sufficiently small norm $\Vert \delta \Vert_{L^2(\omega)}$ (which might also be zero). 
Denote the noisy data by $q^\delta = q + \delta$. Then, 
\begin{enumerate}
\item[(i)] For some sufficiently small mesh size $0<h\leq h_0$ and $\rho>0$ there is a unique $a_h^\delta \in B_\rho(\tilde{a})$ such that $F_h^\delta(a_h^\delta) = 0$. Moreover,
\begin{align}\label{eq:convergence}
\Vert \tilde{a} - a_h^\delta \Vert_{C^2(\Gamma)} \, \lesssim \, h^2 + \Vert \delta \Vert_{L^2(\omega)}\, .
\end{align}
The implicit constant in \eqref{eq:convergence} depends on the dimension of the finite dimensional space $J$ and on $\tilde{a}$.

\item[(ii)] For a suﬀiciently good initial guess the Newton iterates corresponding to the
equation $F_h^\delta(x) = 0$ converge quadratically to the unique root $a_h^\delta$ from part (i).
\end{enumerate}
\end{theorem}

\section{Numerical examples}
In this section we explore some numerical examples that highlight our theoretical findings and rates.
In particular we are interested in verifying the convergence rate that we established in \eqref{eq:convergence}.
For our numerical examples we focus on the geometrical setup, in which $\Omega_2 \subset \R^2$ has a flower shape and $\Omega_1 \subset \Omega_2$ is defined by $\Omega_1 = B_{1/4}(0)$. As before, $\Gamma = \partial \Omega_1$ and $\Gamma_0= \partial \Omega_2$. 
We consider a set of basis functions on $\Gamma$, which are the trigonometric functions $\phi_{1,m}$ and $\phi_{2,n}$ defined by
\begin{align*}
\phi_{1,m-1}(t)\, = \, \frac{2}{\sqrt{\pi}} \cos((m-1)t)\, , \quad \phi_{2,n}(t)\, = \, \frac{2}{\sqrt{\pi}} \sin(nt)\, .
\end{align*}
The finite dimensional space $V_J$ is defined as the span of the functions $\phi_{1,m-1}$ and $\phi_{2,n}$ for $m=1,\dots, J_1$ and $n=1,\dots,J_2$, i.e., $J=J_1 + J_2$. Instead of $V_J$ we also write $V_{J_1,J_2}$. Any function $a \in V_{J_1, J_2}$ can be written as 
\begin{align}\label{eq:arep}
a \, = \, \sum_{m=1}^{J_1}\alpha_m \phi_{1,m-1} + \sum_{n=1}^{J_2} \beta_n \phi_{2,n}\, .
\end{align}
The numerical reconstruction method for recovering $\tilde{a} \in V_{J_1, J_2}$ is Newton's method. 
The variables that we aim to optimize are the parameters $(\alpha_m)_{m=1}^{J_1}$ and $(\beta_n)_{n=1}^{J_2}$ from \eqref{eq:arep}.
For the $k$-th iteration step, the vector of variables that contains the coefficients $\alpha_m$ and $\beta_n$ in succession is denoted by $x_k$.
In our Newton method we additionally incorporate a backtracking line search, which is supposed to stabilize the algorithm somewhat.
In particular, it is supposed to keep the parameter $a$ positive, which is required for the well-posedness of the problems \eqref{eq:nonlin_robin} and \eqref{eq:linRob}.
In detail, when denoting by $a(x_k)$ the current iterate as defined in \eqref{eq:arep} and by $d_k$ the Newton update of the $k$-th iteration step, we determine the smallest number $\kappa \in \N_0$ such that
\begin{align*}
|F(a(x_k+(0.5)^\kappa d_k)) | \, \leq \, |F_h(a(x_{k})| \quad \text{and }  a(x_k+(0.5)^\kappa d_k) >0
\end{align*}
and define the next iterate to be $x_{k+1} = x_k+(0.5)^\kappa d_k$.

For our numerical examples we consider the function $\beta:\R \to \R$ as defined in \eqref{eq:betadef} with $\varepsilon = 1$.
Moreover, we pick the right hand sides $f$ and $g$ given by
\begin{align*}
f(x,y) \, = \, -5xy\mathrm{e}^{\sin(4\pi y)}\, ,\quad (x,y) \in \Omega \quad \text{and } \quad g(x,y) \, = \, 0\, , \quad (x,y) \in \Gamma\, .
\end{align*}
The open set $\omega \subset \Omega$ on which we assume given data is given by six circles located in the petals of the flower shaped domain.
The both discs that lie on the $x$-axis have the radius $0.25$, respectively, the other discs have the radius $0.2$.
The overall size of $\omega$ covers around 29\% of the area of $\Omega$.

All finite element simulations are carried out on a polygonal domain $\Omega_h \approx \Omega$ by the open source computing platform FEniCS.
The parameter $\tilde{a}$, which we aim to reconstruct is given by a linear combination as in \eqref{eq:arep} with $J_1=J_2=6$.
In our numerical examples we also use $J_1=J_2=6$, i.e., the finite dimensional space $V_J$ has the dimension 12.
As a reference solution to compute $q = u^{(\tilde{a})}|_\omega$, we employ a second order finite element method with the maximal mesh size $h \approx 8 \times 10^{-4}$, which is a smaller mesh size than the one we use for the reconstruction. In particular, we highlight that the finer mesh is not a structured uniform refinement of the coarser mesh.
This is supposed to avoid inverse crime.
A visualization of the geometric setup together with the reference solution on $\Omega$ and on $\omega$, is given in Figure~\ref{fig:ex1}.

\textbf{Example 1.}
In our first numerical example the aim is to visually explore the reconstruction of the parameter $\tilde{a}$.
We start with the constant initial guess $a = 2$, use a mesh with the maximal mesh size $h \approx 1.4\times 10^{-3}$ and start the Newton scheme.
\begin{figure}[t!]
\centering 
\raisebox{0.mm}{\includegraphics[scale=.275]{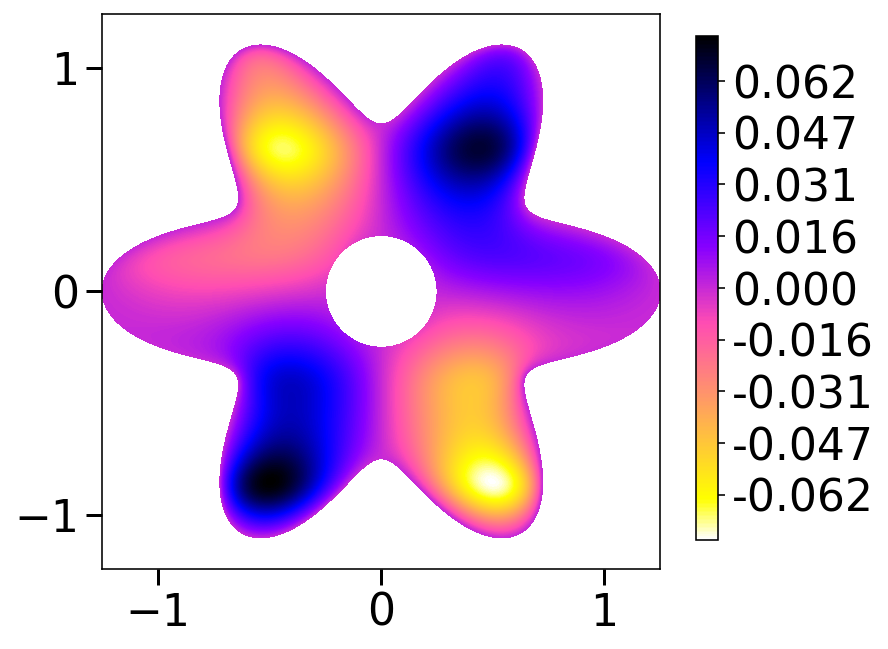}}\hfill
\raisebox{0.mm}{\includegraphics[scale=.275]{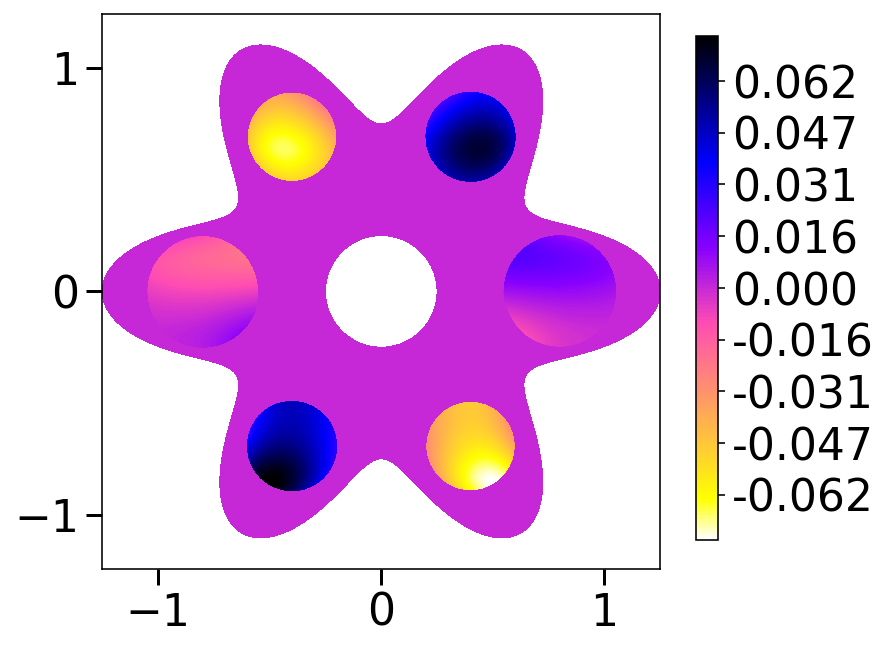}}\hfill
\includegraphics[scale=.23]{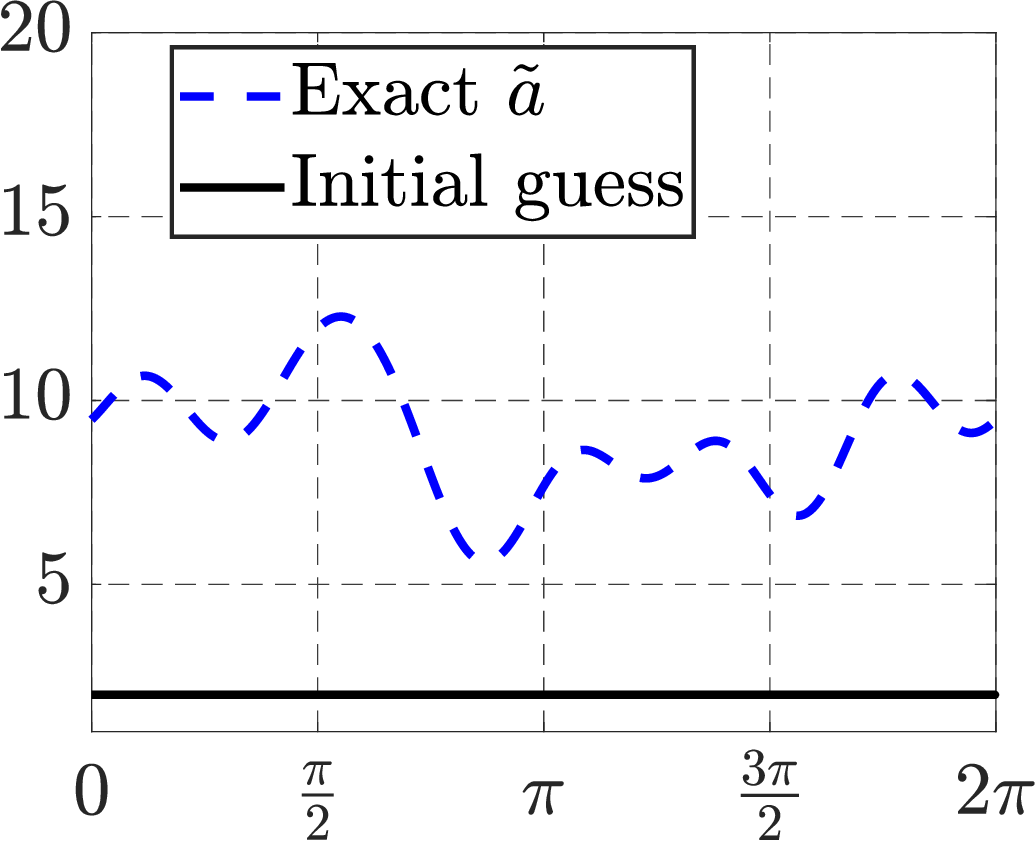}\hfill \\
\includegraphics[scale=.23]{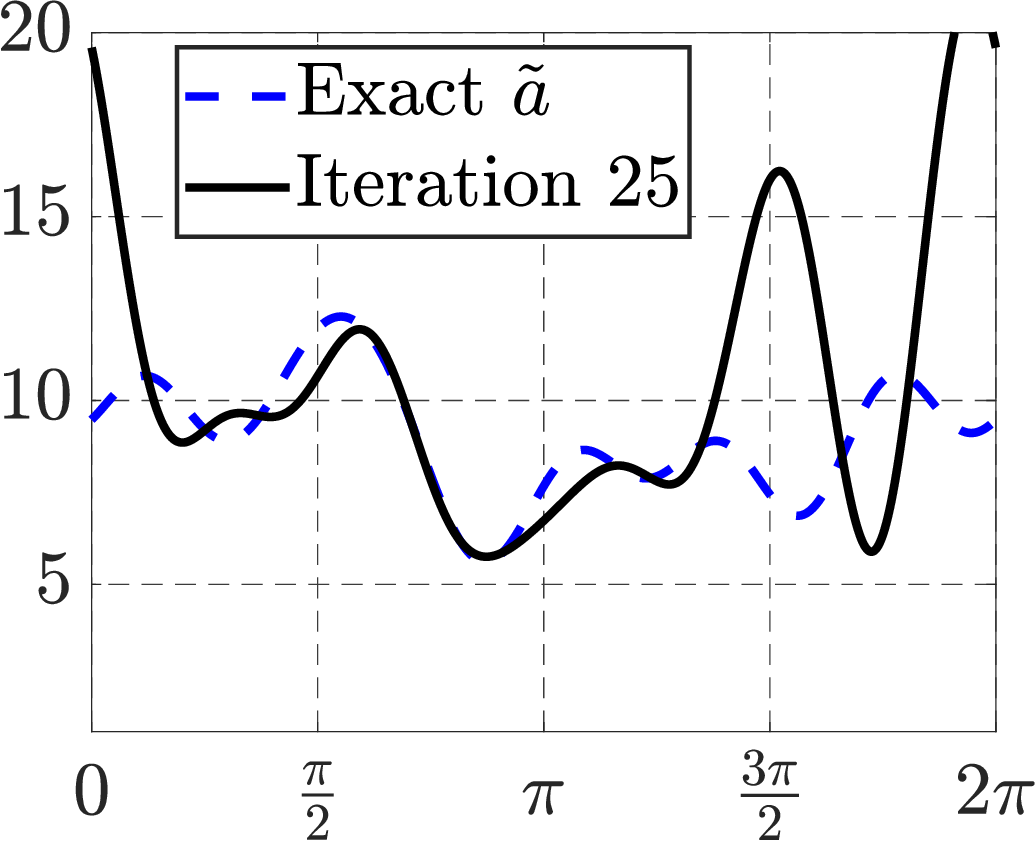}\hfill
\includegraphics[scale=.23]{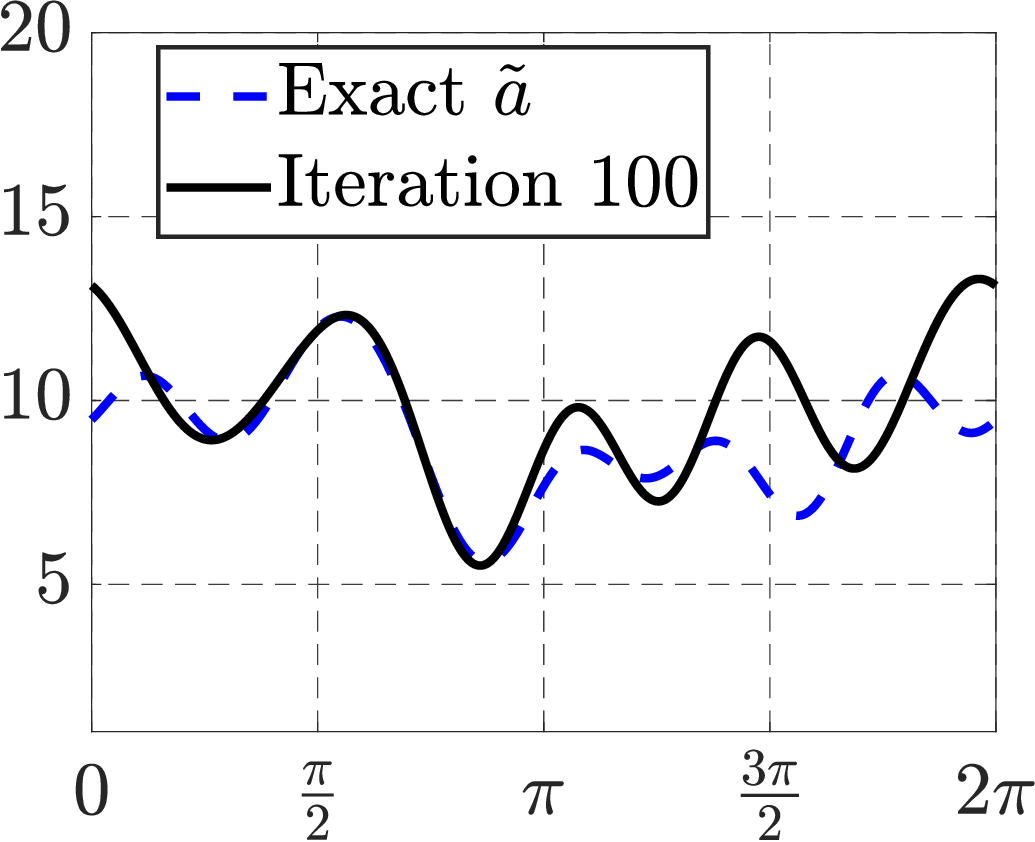}\hfill
\includegraphics[scale=.23]{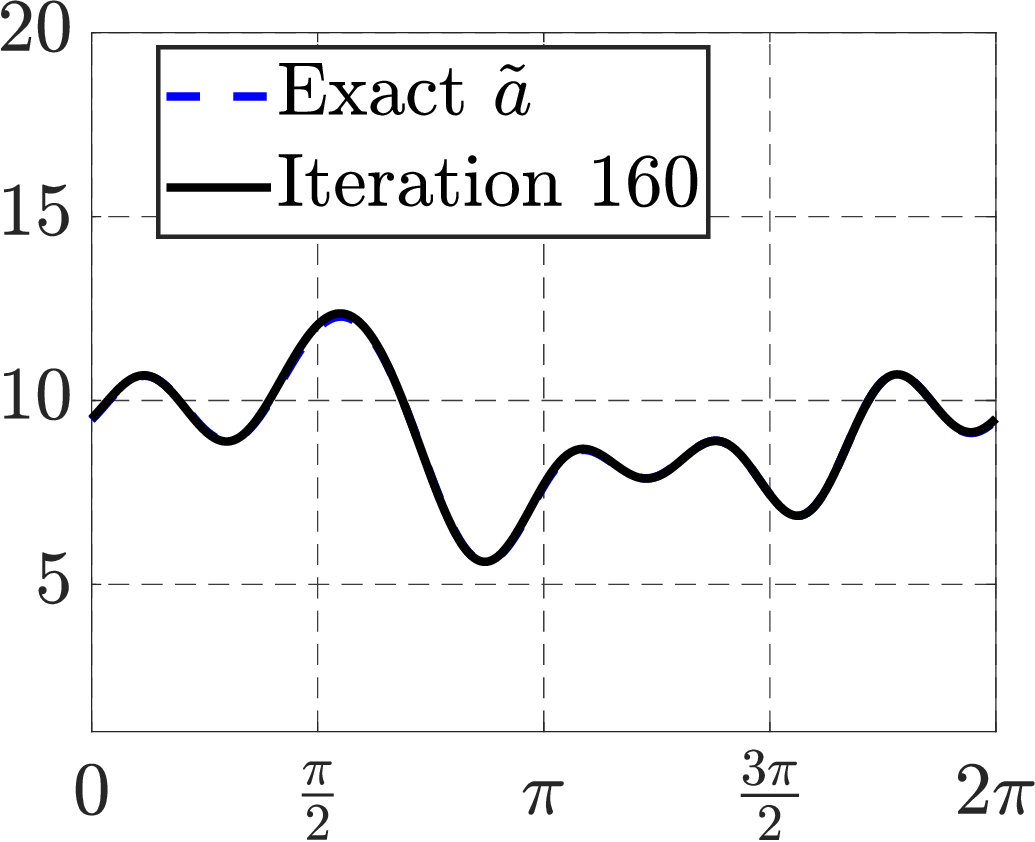}\hfill
\caption{Top left: Visualization of the exact solution on the domain $\Omega$. Top middle: Visualization of $q$ on the domain $\omega$. Top right to bottom right: The initial guess, the iterates 25, 100 and the final iteration 160. }
\label{fig:ex1}
\end{figure}
The convergence history with snapshots of the initial guess, the iterates at steps $\ell = 25,100$ and the final step $\ell = 160$ is found in Figure~\ref{fig:ex1}.
Even though the initial guess appears to be (visually) far off the unknown parameter $\tilde{a}$, the Newton scheme still converges.
We stress nevertheless that a successful reconstruction depends on the initial guess, the exact parameter $\tilde{a}$, the mesh size $h$ and on the domain $\omega$. The latter influences the reconstruction not only via its overall area but also how it is distributed within $\Omega$.

\textbf{Example 2.}
In our second example we numerically validate the convergence rate from Theorem~\ref{thm:main} both for exact and noisy given data.
To generate noisy data we consider the Hadamard type function
\begin{align*}
\delta_j(x) \, = \, \mathrm{Re}(\mathrm{e}^{\mathrm{i}(x-x_j)\cdot z})1_{B_{1/4}(x_j)}(x) \quad\text{with } \quad
z = [10,\; 10\mathrm{i}]^\top\, , x_j = [(-1)^j0.8, \; 0]^\top
\end{align*}
for $j=1,2$. The points $x_j$ correspond to the two middle points of the discs of radius $0.25$ of $\omega$ that lie on the $x$-axis (see also Figure~\ref{fig:ex1}).
We pick some noise level $\sigma\geq 0$ and define noisy data $q^\delta = q^{\delta(\sigma)}$ via
\begin{align*}
q^{\delta(\sigma)} \, = \, q + \sigma( \delta_1 + \delta_2)\, .
\end{align*}
The choice $\sigma = 0$ corresponds to given exact data $q$.
Next, we choose some maximal mesh sizes $h$ to vary within the range between $0.0138$ and $0.0011$.
For a given noise level $\sigma\geq 0$ we start the simulation with the coarsest mesh size and a constant initial guess given by $a=2$.
We run the Newton scheme and measure the relative error
between the final iterate $a_h^\delta$ and $\tilde{a}$ via
 \begin{align}\label{eq:errrel}
e_{\mathrm{rel}}^\delta(h) \, = \, \frac{\Vert \tilde{a} - a_h^\delta\Vert_{C^2(\Gamma)}}{\Vert \tilde{a}\Vert_{C^2(\Gamma)}} \, .
\end{align}
For $\sigma=0$ we simply write $e_{\mathrm{rel}}$. 
For the subsequent mesh sizes we
initialize the solver with the final iterate obtained for the preceding mesh size. 
This procedure is repeated for all mesh sizes in decreasing order.
For our simulations we pick the noise levels 
\begin{align*}
\sigma \in \left\{0,\;1 \times 10^{-7},\; 3\times 10^{-7},\; 1 \times 10^{-6},\; 3\times 10^{-6} \right\}\, .
\end{align*}
For exact data (i.e. for $\sigma = 0$) we expect the relative error to decay as $h^2$.
On the other hand, for noisy data we expect to see the decay $h^2$ until the noise dominates the error bound.
Both phenomena can be found in Figure~\ref{fig:ex2}.

\begin{figure}[t!]
\centering 
\raisebox{2.5mm}{\includegraphics[scale=.36]{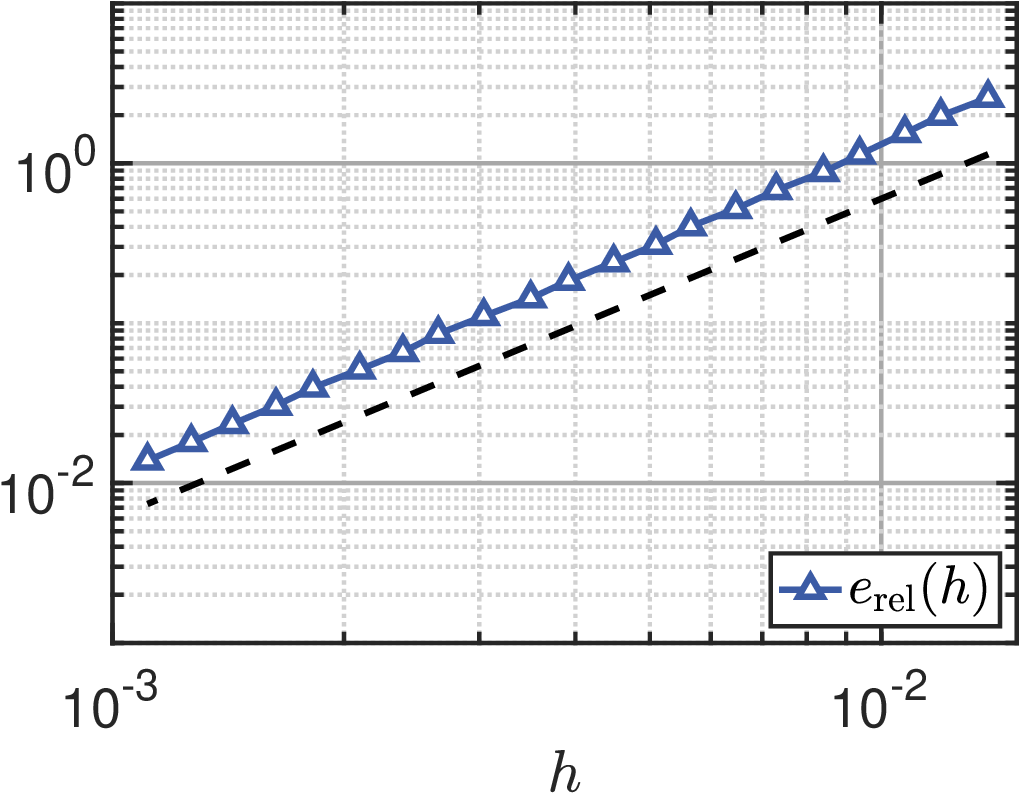}}\hfill
\raisebox{2.5mm}{\includegraphics[scale=.36]{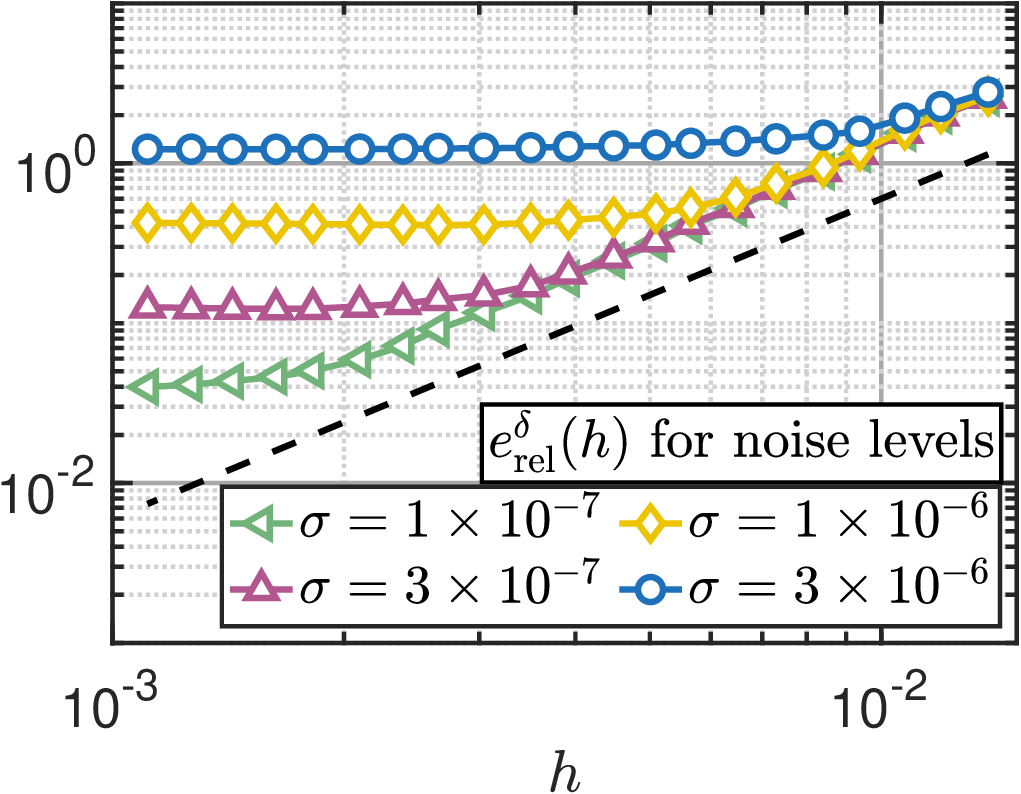}}
\caption{Double logarithmic plots showing the maximal mesh size $h$ against the relative error from \eqref{eq:errrel}. The dashed line has slope 2.
Left: Unperturbed data $q$ is given. Right: Noisy data $q^\delta$ for different noise levels $\sigma$ are given.}
\label{fig:ex2}
\end{figure}

\section*{Appendix}
\begin{proof}[Proof of Thm.\@ \ref{thm:howp}]
The higher-order regularity result of \cite[Thm.\@ I.9, p.\@ 40]{Bre72} yields that the solution of \eqref{eq:nonlin_robin_weak} satisfies $u \in H^2(\Omega)$ and
\begin{align*}
\Vert u \Vert_{H^2(\Omega)} \, \lesssim \, \Vert f \Vert_{L^2(\Omega)} + \Vert u \Vert_{H^1(\Omega)}\, .
\end{align*}
The $H^1$-bound from \eqref{eq:ubound} now yields \eqref{eq:H2bound}. Since $u \in H^2(\Omega)$ we get that the trace $u|_{\Gamma} \in H^{3/2}(\Gamma)$.
We will now show that $\beta(u|_{\Gamma}) \in H^{3/2}(\Gamma)$ by using bounds for composition operators on Sobolev spaces.
For the definition and properties of $C^3$ domains we refer to \cite[Def.\@ 3.28]{McLean00}.
There is a family of open sets $W_j$ covering $\Gamma$ and (rotated and shifted) $C^3$ hypographs $\Omega^{(j)}$ satisfying $W_j \cap \Gamma = W_j \cap \partial \Omega^{(j)}$.
Moreover, let $\Gamma_j = \partial \Omega^{(j)}$.
We consider a partition of unity $(\psi_j, W_j)_{j=1}^N$ for $\Gamma$ such that $\sum_{j=1}^N \psi_j = 1$ on $\Gamma$ and $\psi_j \in C_c^\infty(W_j)$.
By definition (see e.g.\@ \cite[Eq.\@ (3.29)]{McLean00}),
\begin{align}\label{eq:defsobolev}
\Vert \beta(u|_{\Gamma}) \Vert_{H^{3/2}(\Gamma)}^2 \, = \, \sum_{n=1}^N \Vert \psi_j \beta(u|_{\Gamma}) \Vert_{H^{3/2}(\Gamma_j)}^2 \, .
\end{align}
Every $\Gamma_j$ is the boundary of a $C^3$ hypograph up to some rigid motion $\kappa_j : \R^n \to \R^n$. 
Also, let $\varphi_j : \R^{n-1} \to \R^{n}$, $\varphi_j(x') = (x', \zeta_j(x'))$ with $\zeta_j \in C^3(\R^{n-1})$ be the function that describes the boundary of the $C^3$ hypograph $\kappa_j(\Gamma_j)$. This implies that 
\begin{align*}
\kappa_j(\Gamma_j) \, = \, \{ x \in \R^n \, : \, x_n = \zeta_j(x'), \; x' \in \R^{n-1} \} \, 
\end{align*} 
or equivalently, we can write $\Gamma_j = \kappa_j^{-1} \circ \varphi_j$ on $\R^{n-1}$ for every $j=1, \dots, N$.
For every $j$ let $\xi_j \in C_c^\infty(W_j)$ with $\xi_j = 1$ for $x \in \supp(\psi_j)$.
Then, 
\begin{align*}
\psi_j \beta(u|_{\Gamma}) \, = \, \psi_j \beta(\xi_j u|_{ \Gamma}) \quad \text{on } \Gamma_j \, .
\end{align*}
We recall the definition of the $H^s$-norm on shifted and rotated hypographs from \cite{McLean00}. In this work the definition is given for $0\leq s \leq 1$ on p.\@ 98 and extended for larger $s$ on p.\@ 99. We find that
\begin{align*}
\Vert \psi_j \beta(\xi_j u|_{\Gamma}) \Vert_{H^{3/2}(\Gamma_j)} \, = \, \Vert ( \psi_j \beta(\xi_j u|_{\Gamma}) ) \circ (\kappa_j^{-1} \circ \varphi_j) \Vert_{H^{3/2}(\R^{n-1})} \, .
\end{align*} 
For abbreviation, we write $\widetilde{\psi}_j = \psi_j \circ (\kappa_j^{-1} \circ \varphi_j)$, $\widetilde{\xi}_j = \xi_j \circ (\kappa_j^{-1} \circ \varphi_j)$ and $\widetilde{u}_j = u|_{\Gamma}\circ (\kappa_j^{-1} \circ \varphi_j)$.
Since $\widetilde{\psi}_j \in C^2$ (in fact it is even $C^3$), we can use \cite[Thm.\@ 3.20]{McLean00} and see that
\begin{align*}
\Vert \widetilde{\psi}_j \beta(\widetilde{\xi}_j \widetilde{u}_j) \Vert_{H^{3/2}(\R^{n-1})} \, \lesssim \, \Vert \beta(\widetilde{\xi}_j \widetilde{u}_j) \Vert_{H^{3/2}(\R^{n-1})} \, .
\end{align*}
We introduce yet another cutoff $\chi_j \in C_c^{\infty}(W_j)$ that satisfies $\chi_j =1 $ for $x \in \widetilde{\xi}_j(\R^{n-1}) \widetilde{u}_j(\R^{n-1})$.
Then, we find that $\beta(\widetilde{\xi}_j \widetilde{u}_j) = \chi_j(\widetilde{\xi}_j \widetilde{u}_j)\beta(\widetilde{\xi}_j \widetilde{u}_j)$ and applying
\cite[Thm.\@ 11]{BoSi11} to $\widetilde{\beta} = \chi_j \beta$ now yields that
\begin{align*}
\Vert \beta(\widetilde{\xi}_j \widetilde{u}_j) \Vert_{H^{3/2}(\R^{n-1})} \, \lesssim \, \Vert \beta \Vert_{C^2(u(\Gamma))} \Vert \widetilde{\xi}_j \widetilde{u}_j \Vert_{H^{3/2}(\R^{n-1})} \Big(1 + \Vert \widetilde{\xi}_j \widetilde{u}_j \Vert_{H^{3/2}(\R^{n-1})}^{1/2} \Big) \, .
\end{align*}
Since $\widetilde{\xi}_j$ localizes to $W_j$ we can now estimate
\begin{align*}
\Vert \beta(\widetilde{\xi}_j \widetilde{u}_j) \Vert_{H^{3/2}(\R^{n-1})} \, \lesssim \, \Vert\beta \Vert_{C^2(u(\Gamma))} \Vert u \Vert_{H^{3/2}(\Gamma)} \big(1 + \Vert u \Vert_{H^{3/2}(\Gamma)}^{1/2} \big) \, .
\end{align*}
Returning to \eqref{eq:defsobolev} yields that
\begin{align}\label{eq:betabound}
\Vert \beta(u|_{\Gamma}) \Vert_{H^{3/2}(\Gamma)} \, \lesssim \, \Vert\beta \Vert_{C^2(u(\Gamma))} \Vert u \Vert_{H^{3/2}(\Gamma)} \big(1 + \Vert u \Vert_{H^{3/2}(\Gamma)}^{1/2} \big)\, .
\end{align}
Since $f \in H^1(\Omega)$ is assumed and $\beta(u|_{\Gamma}) \in H^{3/2}(\Gamma)$ a higher order regularity result as described, e.g., in \cite[Thm.\@ 2.5.1.1]{Gris85} yields that $u \in H^3(\Omega)$. 
The closed graph theorem in combination with \cite[Thm.\@ 3.20]{McLean00}, the $H^2$ bound of $u$ in \eqref{eq:H2bound} and \eqref{eq:betabound} shows that
\begin{align*}
\Vert u \Vert_{H^3(\Omega)} \, &\lesssim \, \Vert f \Vert_{H^1(\Omega)} + \Vert g \Vert_{H^{3/2}(\Gamma)}  + \Vert a \beta(u|_{\Gamma}) \Vert_{H^{3/2}(\Gamma)} \\
& \lesssim \, \Vert f \Vert_{H^1(\Omega)} + \Vert g \Vert_{H^{3/2}(\Gamma)}  \\
&\phantom{\lesssim \, }+ \Vert a \Vert_{C^2(\Gamma)}\Vert\beta \Vert_{C^2(u(\Gamma))} \Vert u \Vert_{H^{3/2}(\Gamma)} \big(1 + \Vert u \Vert_{H^{3/2}(\Gamma)}^{1/2} \big) \\
& \lesssim \, \Vert f \Vert_{H^1(\Omega)}  + \Vert g \Vert_{H^{3/2}(\Gamma)}  \\
&\phantom{\lesssim \, }+\Vert a \Vert_{C^2(\Gamma)}\Vert\beta \Vert_{C^2(u(\Gamma))} \Vert u \Vert_{H^2(\Omega)}\big(1 + \Vert u \Vert_{H^2(\Omega)}^{1/2} \big) \, .
\end{align*}
\end{proof}
\bibliographystyle{abbrvurl}
\bibliography{references}

\end{document}